%% file: qhp_main.tex
\documentclass[authorcolumns,numberwithinsect,a4paper]{no-lipics} 
\usepackage{ragged2e}
\usepackage{cleveref} 
\usepackage{lipsum}

\input{macros}

\title{The Complexity of Color-constrained Paths in Semicomplete Multipartite Digraphs}
\author{Julian Christoph Brinkmann}{Research Group for Theoretical Computer Science\\Goethe University Frankfurt\\60629 Frankfurt am Main Germany}{J.Brinkmann@em.uni-frankfurt.de}{https://orcid.org/0009-0000-0332-4543}{}
\authorrunning{J. Brinkmann}

\begin{document}
	
\maketitle
\pagenumbering{roman}
\begin{abstract}
	\input{abstract.tex}
\end{abstract}

\thispagestyle{plain}
\tableofcontents
\clearpage
\pagenumbering{arabic}

\input{main_body.tex}

\bibliographystyle{elsarticle-num}
\bibliography{qhp}
\end{document}

%% file: macros.tex
\newcommand{\NP}{\ensuremath{\mathsf{NP}}}
\newcommand{\Oh}{\mathcal{O}}
\newcommand{\oh}{o}

\newcommand{\N}{\ensuremath{\mathbb{N}}}
\newcommand{\Z}{\ensuremath{\mathbb{Z}}}

\newcommand{\trSAT}{\ensuremath{\mathsf{3\text{-}SAT}}}
\newcommand{\twSAT}{\mathsf{2\text{-}SAT}}
\newcommand{\SAT}{\mathsf{SAT}}
\newcommand{\constSAT}[1]{\mathsf{SAT}(S)}
\newcommand{\RCP}{\mathsf{CCP}}

\newcommand{\C}{\mathcal{C}}
\newcommand{\majority}{\mathsf{majority}}
\newcommand{\minority}{\mathsf{minority}}

%% file: abstract.tex
Every semicomplete multipartite digraph contains a quasi-Hamiltonian path, but deciding the presence of a quasi-Hamiltonian path with prescribed start and end vertex is~$\NP$-complete.
Recently, Bang-Jensen, Wang and Yeo\ (Discrete Appl. Math. 2025) showed that deciding the presence of a quasi-Hamiltonian cycle which does not contain at least one vertex from each color class is~$\NP$-complete.
Similarly, deciding the presence of a quasi-Hamiltonian cycle which intersects every part exactly once is also~$\NP$-complete as shown in the same work.

In this paper, we study the problem of finding a path with prescribed start and end vertex that contains at least~$a$ and at most~$b$ vertices from each color class in a semicomplete multipartite digraph where all color classes have size~$\alpha$.
This unifies the Hamiltonian path problem, the quasi-Hamiltonian path problem and the path-version of the cycle problems mentioned above, among other problems.
Using Schaefer's dichotomy theorem, we classify the complexity of almost all problems in this framework.
The open problems are the Hamiltonian path problem on semicomplete multipartite digraphs and the quasi-Hamiltonian path problem restricted to semicomplete multipartite digraphs with independence number two.

We then investigate semicomplete multipartite digraphs with independence number two from a structural perspective.
We generalize sufficient criteria for Hamiltonian~$(s,t)$-paths in semicomplete digraphs to sufficient criteria for quasi-Hamiltonian~$(s,t)$-paths in this class.
Although this does not settle the problem, our initial results suggest that this special case may be solvable in polynomial time.

%% file: main_body.tex
\section{Introduction}
A semicomplete multipartite digraph~$D$ is a simple directed graph~$D$ with color classes, also called parts or partite sets,~$V_1,\dots,V_c$ such that there exists at least one arc between a pair of vertices if and only if the vertices are in different color classes.
The sets~$V_1,\dots,V_c$ are the unique maximal independent sets in \(D\), so the independence number satisfies~$\alpha(D)=\max\{|V_1|,\dots,|V_c|\}$ and~\(\chi(D)=c\).
A multipartite tournament is a semicomplete multipartite digraph without cycles of length two.
A semicomplete digraph~$D$ is a semicomplete multipartite digraph with~$\alpha(D)=1$, i.e.~every color class consists of a single vertex.
A tournament is a semicomplete digraph without cycles of length two.
A quasi-Hamiltonian path (resp. cycle) in~$D$ is a path (resp. cycle) that contains at least one vertex from each color class.
For semicomplete digraphs, quasi-Hamiltonian paths are Hamiltonian paths.
Similarly, quasi-Hamiltonian cycles are Hamiltonian cycles.
An~$(x,y)$-path is a path with start vertex~$x$ and end vertex~$y$.
We will often abbreviate quasi-Hamiltonian path by QHP and semicomplete multipartite digraph by SMD.

Rédei~\cite{redei1934kombinatorischer} shows that the number of Hamiltonian paths in a tournament is odd.
In particular, every tournament contains a Hamiltonian path.
Similarly, every semicomplete multipartite digraph contains a quasi-Hamiltonian path.
While not every semicomplete multipartite digraph contains a Hamiltonian path, the longest path problem restricted to semicomplete multipartite digraphs can be solved in time~$\Oh(n^\beta)$ with an algorithm due to Gutin~\cite{DBLP:journals/siamdm/Gutin93}.
Here, the value~$\beta$ is the exponent of the running time of the fastest algorithm for computing a minimum cost maximum bipartite matching.
The best currently known value, \(\beta\leq2+\oh(1)\), is due to Chen et.~al.~\cite{DBLP:journals/cacm/ChenKLPGS23}.
As the longest path problem subsumes the Hamiltonian path problem, the same algorithm solves the Hamiltonian path problem restricted to semicomplete multipartite digraphs.
The complexity of the Hamiltonian path problem with prescribed endpoints restricted to semicomplete multipartite digraphs remains open.

In semicomplete digraphs, the presence of an~$(x,y)$-Hamiltonian path can be decided in time~$\Oh(n^5)$ using the algorithm by Bang-Jensen, Manoussakis and Thomassen~\cite{DBLP:journals/jal/Bang-JensenMT92}.
For semicomplete multipartite digraphs, the presence of a quasi-Hamiltonian path whose endpoints are in the set~$\{x,y\}$ can be decided in time \(\Oh(n^{27})\) using the algorithm by Bang-Jensen, Maddaloni, and Simonsen~\cite{DBLP:journals/dam/Bang-JensenMS13}.
However, when the direction of the path is made explicit as well, the problem becomes~$\NP$-complete as shown in the same work.
Although this is not stated in~\cite[Theorem~3.1]{DBLP:journals/dam/Bang-JensenMS13}, the proof shows that the problem remains hard even when restricting the inputs to SMDs with independence number exactly three.
This is due to the fact that the reduction is from Boolean satisfiability and, apart from auxiliary singleton color classes, the color classes in the reduction correspond to the literals in clauses of the given CNF-formula.
Thus, color classes of size three are sufficient for \(\NP\)-completeness when reducing from~$\trSAT$.

Recently, Bang-Jensen, Wang and Yeo~\cite{DBLP:journals/dam/Bang-JensenWY25} proved a similar result:
Deciding the presence of a quasi-Hamiltonian cycle~$C$ satisfying the inequality~$1\leq|V(C)\cap V_i|<|V_i|$ for all~$i$ such that~$1\leq i\leq c$, i.e.~a quasi-Hamiltonian cycle which avoids at least one vertex from each color class, is~$\NP$-complete.
Similarly, deciding the presence of a quasi-Hamiltonian cycle where the equality~$|V(C)\cap V_i|=1$ is satisfied for all~$i$ such that~$1\leq i\leq c$, i.e.~a quasi-Hamiltonian cycle which intersects each color class exactly once, is~$\NP$-complete.

These results motivate us to consider path problems with prescribed start and end vertices under cardinality constraints on~$|V(P)\cap V_i|$ for each~$i$ such that~$1\leq i\leq c$.
Concretely, we study the setting where constraints take the form~$a\leq|V(P)\cap V_i|\leq b$ for all~$i$ with~$1\leq i\leq c$, and all color classes have size exactly~$\alpha$.
We call this problem the \((a,b,\alpha)\)-color-constrained path problem, abbreviated as~$(a,b,\alpha)\text{-}\RCP$.
A formal definition is given in \cref{sec:covering-path}.
This problem contains as special cases the Hamiltonian path problem, the quasi-Hamiltonian path problem and the path-analogues of the cycle problems considered in~\cite{DBLP:journals/dam/Bang-JensenWY25}.
Using Schaefer's dichotomy theorem~\cite{DBLP:conf/stoc/Schaefer78}, we are able to settle the complexity of~$(a,b,\alpha)\text{-}\RCP$ for almost all choices of the parameters~$a$,~$b$ and~$\alpha$, proving the following:

\begin{restatable*}{mtheoremq}{rstMtOne}\label{thm:general-rcp-hardness}
	Let~$(a,b,\alpha)\in\Z^3$ satisfy~$0\leq a\leq b\leq\alpha$ and~$b\geq1$.
	If~$(a,b)=(0,\alpha)$ or \({\alpha=1}\), the problem~$(a,b,\alpha)\text{-}\RCP$ can be solved in polynomial-time.
	If~${(a,b)\notin\{(0,\alpha),(\alpha,\alpha)}\}$ and~$\alpha\geq3$, the problem~$(a,b,\alpha)\text{-}\RCP$ is \NP-complete.
\end{restatable*}

Thus, the unclassified cases are~$a=b=\alpha$, which corresponds to the Hamiltonian path problem, and the cases where~$\alpha=2$ with~$(a,b)\neq(0,2)$.
While semicomplete bipartite digraphs have received special attention in the literature, e.g.~in~\cite{DBLP:journals/jal/Yeo99}, we are unaware of any work that expressly deals with semicomplete multipartite digraphs of independence number at most~2.
To the best of our knowledge, this is the first work studying this particular class.
We generalize sufficient conditions for~$(x,y)$-Hamiltonian paths in semicomplete digraphs from~\cite{DBLP:journals/jal/Bang-JensenMT92} to sufficient conditions for \((x,y)\)-quasi-Hamiltonian paths in semicomplete multipartite digraphs.
Concretely, we obtain the following results:

\begin{restatable*}{mtheorem}{rstMtTwo}\label{thm:not-2-strong-without-xy}
	Let~$D=(V,A)$ be a 2-strong semicomplete multipartite digraph with distinct vertices~${x,y,z\in V}$ such that neither~$x$ nor~$y$ is contained in a~2-cycle and~$D-x$ as well as~$D-y$ are~2-strong.
	If~$D-\{x,y,z\}$ is not strong and~$\chi(D-\{x,y,z\})\geq4$, then~$D$ contains an~$(x,y)$-quasi-Hamiltonian path.
\end{restatable*}

\begin{restatable*}{mtheorem}{rstMtThree}\label{thm:2-sep-all-trivial}
	Let~$D=(V,A)$ be a 2-strong semicomplete multipartite digraph such that~$\alpha(D)\leq 2$ and~${|V|\geq 5+5\alpha(D)}$.
	Let~$x,y\in V$ be different vertices of~$D$ such that neither~$x$ nor~$y$ is contained in a 2-cycle,~$y\Rightarrow x$, and~$D-x$ as well as~$D-y$ are~2-strong.
	If all~2-separators of~$x$ and~$y$ (if any) are trivial, then~$D$ contains an~$(x,y)$-quasi-Hamiltonian path.
\end{restatable*}

\Cref{thm:not-2-strong-without-xy} generalizes~\cite[Theorem~3.3]{DBLP:journals/jal/Bang-JensenMT92}.
\Cref{thm:2-sep-all-trivial} generalizes~\cite[Theorem~3.4]{DBLP:journals/jal/Bang-JensenMT92}.
As~\cref{thm:2-sep-all-trivial} assumes the independence number to be at most 2, the problem of finding a QHP with prescribed start and end vertex restricted to this class may be easier than the general case.

Bang-Jensen, Manoussakis and Thomassen~\cite{DBLP:journals/jal/Bang-JensenMT92} use the corresponding conditions, among other results, to construct a polynomial-time algorithm for the~$(x,y)$-Hamiltonian path problem on semicomplete digraphs.
While the general \((x,y)\)-quasi-Hamiltonian path problem is~$\NP$-complete, the proof from~\cite{DBLP:journals/dam/Bang-JensenMS13} only works when allowing instances with~$\alpha(D)\geq3$.
The reduction is from~$\trSAT$ and the independence number corresponds to the number of different literals allowed in a clause of the input formula.
As~$\twSAT$ is polynomial-time solvable, this does not prove hardness for instances with~$\alpha(D)\leq2$.
Of course, the problem may still be~$\NP$-complete through a currently unknown reduction.
Settling the complexity of this case is an interesting direction for future work.

We also prove that the existence of a Hamiltonian path is equivalent to unilateral connectivity for SMDs with independence number at most~2, see \cref{thm:hamiltonian-criterion-nsmd}.
This implies that the Hamiltonian path problem restricted to this class can be solved in time~$\Oh(n^2)$, which improves on the fastest previously known algorithm for general SMDs.

\section{Terminology and Previous Work}\label{sec:prev}
For a digraph \(D\), the independence number of \(D\), denoted \(\alpha(D)\), is the independence number of the underlying graph.
Similarly, the chromatic number of \(D\), denoted \(\chi(D)\), is the chromatic number of the underlying graph.
Given a semicomplete multipartite digraph~$D$, we denote the unique partition~$\{V_1,\dots,V_c\}$ underlying~$D$ by~$\C(D)$.
Observe that~$\C(D)$ can be computed in polynomial time by inspecting the adjacency matrix.
For a subset~${U\subseteq V(D)}$, we denote by~$H(U)$ the set of vertices that share a color class with a vertex in~$U$.
Formally,~$H(U)=\{v\in V(D) : v\in W\in\C(D), W\cap U\neq\emptyset\}$.
For vertices~$v\in V(D)$, we write~$H(v)$ instead of~$H(\{v\})$.
We denote the subgraph induced by a subset~$U$ of \(V(D)\) by~$D[U]$.
Formally,~${D[U]=(U,\{vw\in A(D): v,w\in U\})}$.
An~$(s,t)$-path is a path starting at~$s$ and ending in~$t$.
An~$[s,t]$-path has end vertices in~$\{s,t\}$, i.e.~it is either an~$(s,t)$-path or a~$(t,s)$-path.
If the arc~$xy$ exists in~$D$, we say that~$x$ dominates~$y$.
For sets of vertices~$U,W\subseteq V(D)$, we write~$U\Rightarrow W$ to denote that no vertex in~$W$ dominates a vertex in~$U$.
We write~$U\rightarrow W$ if~$U\Rightarrow W$ and~$U\times W\subseteq A(D)$.
Given a path~$P=v_1\dots v_\ell$,~$P[v_i,v_j]$ denotes the subpath from~$v_i$ to~$v_j$ in~$P$, i.e.~$v_iv_{i+1}\dots v_{j-1}v_j$.
Similarly,~$P[v_i,v_j[$ denotes the path~$P[v_i,v_j]-v_j$.
The subpaths~$P]v_i,v_j]$ and~$P]v_i,v_j[$ are defined analogously.
We also use this notation for subpaths in directed cycles.
As the cycles are directed, this is unambiguous.
A digraph is \emph{strong} if it contains an~$(x,y)$-path for every pair~$x,y\in V(D)$ with~$x\neq y$.
A digraph is \emph{$k$-strong} if it contains at least~$k+1$ vertices, is strong and remains strong after removing any set of at most~$k-1$ vertices.
A digraph is \emph{unilaterally connected} if there exists an~$[x,y]$-path for every pair~$x,y\in V(D)$ with~$x\neq y$.
A~\emph{$k$-separator} of vertices~$x$ and~$y$ is a set~$X\subseteq V(D)$ such that~$|X|\leq k$ and there is no~$(x,y)$-path in~$D-X$.
A~$k$-separator of~$x$ and~$y$ is \emph{trivial} if~$N^+(x)\subseteq X$ or~$N^-(y)\subseteq X$.
Given an undirected graph~$G$, an \emph{orientation} of \(G\) is a digraph obtained from~$G$ by replacing each edge with a single arc.
A \emph{biorientation} of \(G\) is a digraph obtained from~$G$ by replacing each edge with an arc or a pair of opposing arcs.

When dealing with path problems in semicomplete multipartite digraphs, it is often helpful to consider the following decomposition:
\begin{lemmaq}[\cite{DBLP:journals/dm/TewesV99}]\label{lem:linear-decomp}
	Let~$D=(V,A)$ be a semicomplete multipartite digraph and let~$V_1,\dots,V_c$ be the color classes of \(D\).
	Then there exists a unique partition of~$V$ into~$R_1,\dots,R_k$, where, for all~$1\leq i\leq k$,~$D[R_i]$ is either a strong component of~$D$ or~$R_i\subseteq V_j$ for some~$1\leq j\leq c$, such that~$R_i\Rightarrow R_j$ for all~$1\leq i<j\leq k$ and there are~$x_i\in R_i$ and~$y_i\in R_{i+1}$ such that~$x_iy_i\in A$ for~$1\leq i<k$.
\end{lemmaq}
\noindent We call this unique partition the linear decomposition of~$D$.
The sets \(R_1,\dots,R_k\) are the linear components of \(D\).
There is an~$\Oh(n^2)$ algorithm to find the linear decomposition of a given SMD~\cite{DBLP:journals/dam/Bang-JensenMS13}:
Find the strong components, find an acyclic ordering of the strong components and then group together all the vertices from the same color class that form consecutive components in the ordering.

When constructing~$(s,t)$-quasi-Hamiltonian paths in SMDs, it is sufficient to find a collection of not necessarily internally disjoint~$(s,t)$-paths that cover all color classes such that the union of all these paths is an acyclic digraph.
This is due to the fact that two paths whose union does not contain a cycle can be combined to obtain a new path that covers at least one vertex from every color class covered by at least one of the paths.

\begin{lemmaq}[{\cite[Lemma 5.1]{DBLP:journals/dam/Bang-JensenMS13}}]\label{lem:merging}
	Let~$D=(V,A)$ be a semicomplete multipartite digraph,~$P$ be an~$(s_1, t_1)$-path and~$Q$ be an~$(s_2, t_2)$-path such that~$P\cup Q$ does not contain a directed cycle and~$X\subseteq V(P)\cup V(Q)$ be a vertex set where each vertex is from a distinct color class.
	Then there exists an $(s_i, t_j)$-path~$R$ in \(D\) such that \({X\subseteq V(R)\subseteq V(P)\cup V(Q)}\), where~$i,j\in\{1, 2\}$.
	Such a path can be found in~$\Oh(|V(P)\cup V(Q)|)$ time.
\end{lemmaq}

Restricted to semicomplete digraphs~$D$, \cref{lem:merging} says that~$D$ is path-mergeable, which is another generalization of tournaments, see also~\cite{DBLP:journals/jgt/Bang-JensenG98}.
Although the path \(R\) does, in general, not contain all vertices of the paths \(P\) and \(Q\), we will refer to this process as merging.
Using \cref{lem:linear-decomp,lem:merging}, it is not hard to prove the following results by considering the structure imposed by the connectivity conditions:

\begin{lemmaq}[{\cite[Theorem 5.2]{DBLP:journals/dam/Bang-JensenMS13}}]\label{lem:cycle-cover}
	Let~$D=(V,A)$ be a strong semicomplete multipartite digraph and let \(X\subseteq V\) be a vertex set where each vertex is from a distinct color class.
	There exists a cycle~$C$ containing~$X$ in~$D$ and such a cycle can be identified in time~$\Oh(n^2)$.
\end{lemmaq}

\begin{lemmaq}[Reformulation of {\cite[Lemma 5.6]{DBLP:journals/dam/Bang-JensenMS13}}]\label{lem:non-strong}
	Let~$D=(V,A)$ be a connected non-strong semicomplete multipartite digraph with~$x,y\in V$, let~$R_1,\dots,R_k$ be its linear decomposition and let~$x\in R_i$,~$y\in R_j$ with~$i<j$.
	There exists an~$(x,y)$-quasi-Hamiltonian path in~$D$ if and only if~$R_i\cup\dots\cup R_j$ intersects all color classes of~$V$.
	In this situation, there is an algorithm computing an~$(x,y)$-quasi-Hamiltonian path or deciding that there is none in time~$\Oh(n^2)$.
\end{lemmaq}

\begin{lemmaq}[{\cite[Lemma 5.8]{DBLP:journals/dam/Bang-JensenMS13}}]\label{lem:not-2-strong}
	Let~$D=(V,A)$ be a strong semicomplete multipartite digraph with different vertices~$x,y\in V$ and let~$D-x$ and~$D-y$ be strong.
	If~$D-\{x,y\}$ is not strong then there exists an~$(x,y)$-quasi-Hamiltonian path and such a path can be identified in time~$\Oh(n^2)$.
\end{lemmaq}

The following theorem summarizes Lemma 5.7, Lemma 5.8, Theorem 5.9 and Corollary 5.12 from~\cite{DBLP:journals/dam/Bang-JensenMS13} to characterize the existence of~$[x,y]$-quasi-Hamiltonian paths in sufficiently colorful SMDs.
A complete, more complex characterization of weakly quasi-Hamiltonian-connected SMDs was given in~\cite{DBLP:journals/dam/GuoLS12}, but \cref{thm:weak-qhp} is sufficient for our purposes.
\begin{theoremq}\label{thm:weak-qhp}
	Let~$D$ be a strong semicomplete multipartite digraph and~$x$,~$y$ distinct vertices of~$D$ such that~$\chi(D-\{x,y\})\geq 5$.
	Then~$D$ has an~$[x,y]$-quasi-Hamiltonian path unless one of the following two conditions is satisfied, in which case there is no~$[x,y]$-quasi-Hamiltonian path:
	\begin{enumerate}[label=(\arabic*)]
		\item ~$D-x$ is not strong with linear decomposition~$R_1,\dots,R_k$,~$y\in R_j$ and neither~$R_1\cup\dots\cup R_j$ nor~$R_j\cup\dots\cup R_k$ covers all color classes of~$D-H(x)$.
		\item ~$D-y$ is not strong with linear decomposition~$R_1,\dots,R_k$,~$x\in R_i$ and neither~$R_1\cup\dots\cup R_i$ nor~$R_i\cup\dots\cup R_k$ covers all color classes of~$D-H(y)$.\qedhere
	\end{enumerate}
\end{theoremq}

When trying to find~$[s,t]$-quasi-Hamiltonian paths, it can be helpful to consider quasi-Hamiltonian cycles in~$D-\{s,t\}$.
If the cycle is sufficiently large and obeys a particular structure, the following lemma can be used to obtain the desired path:
\begin{lemmaq}[{\cite[Lemma 5.11]{DBLP:journals/dam/Bang-JensenMS13}}]\label{lem:2-out-of-3-path}
	For any three distinct vertices of a strong semicomplete multipartite digraph, there exists a quasi-Hamiltonian path connecting two of them.
\end{lemmaq}

The following theorem generalizes~\cite[Theorem 5.4]{DBLP:journals/jct/Thomassen80a}, which was used in the algorithm for the~$(s,t)$-Hamiltonian path problem in~\cite{DBLP:journals/jal/Bang-JensenMT92}.
In \cref{sec:structural-results}, we will generalize the two main remaining structural results used in~\cite{DBLP:journals/jal/Bang-JensenMT92}.
\begin{theoremq}[{\cite[Theorem~3.2]{DBLP:journals/dam/GuoLS12}}]\label{thm:3-paths}
	Let~$D=(V,A)$ be a 2-strong semicomplete multipartite digraph and let~$x,y\in V$ be distinct vertices such that there exist three internally disjoint~$(x,y)$-paths of length at least 2 in~$D$.
	Then~$D$ contains an~$(x,y)$-quasi-Hamiltonian path.
\end{theoremq}

An important tool in our \NP-hardness proof is Schaefer's dichotomy theorem \cite[Theorem~2.1]{DBLP:conf/stoc/Schaefer78}.
In modern terms, this theorem classifies constrained satisfaction problems into polynomial-time solvable and \NP-complete cases based on the structure of the allowed relations, see \cite{DBLP:journals/csur/Chen09} for a more detailed discussion.
We present a slightly simplified version of Schaefer's result that is tailored to our purposes, see \cref{thm:dichotomy}.
Our presentation follows \cite{DBLP:journals/csur/Chen09}.
We begin with the required terminology.

Let \(\emptyset\neq S=\{R_1,\dots,R_\ell\}\) be a set of ternary Boolean relations.
The problem \(\constSAT{S}\) is defined as follows:
Given Boolean variables \(X_1,\dots,X_n\) and constraints \(R_{i_{j,0}}(X_{i_{j,1}},X_{i_{j,2}},X_{i_{j,3}})\), where \(1\leq i_{j,0}\leq\ell\) and \({1\leq i_{j,1}, i_{j,2}, i_{j,3}\leq n}\) for all \(1\leq j\leq m\), does there exist an assignment~\({f:\{X_1,\dots,X_n\}\to\{0,1\}}\) such that \((f(X_{i_{j,1}}),f(X_{i_{j,2}}),f(X_{i_{j,3}}))\in R_{j,0}\) for all \(1\leq j\leq m\)?

A polymorphism of a relation \(R\subseteq\{0,1\}^3\) is a function \(f:\{0,1\}^k\to\{0,1\}\) such that
for all~\({r_1,r_2,r_3\in R}\) it holds that the triple \((f(r_{1,1},r_{2,1},r_{3,1}),f(r_{1,2},r_{2,2},r_{3,2}),f(r_{1,3},r_{2,3},r_{3,3}))\) is in~\(R\), where~\({r_i=(r_{i,1},r_{i,2},r_{i,3})}\) for all \({i\in\{1,2,3\}}\).

Schaefer's dichotomy theorem allows classifying the complexity of \(\constSAT{S}\) based on whether at least one of six particular functions is a polymorphism of all relations in \(S\):

\begin{theoremq}[Simplified Dichotomy Theorem]\label{thm:dichotomy}
	Let \(S=\{R_1,\dots,R_\ell\}\) be a non-empty set of ternary Boolean relations.
	Then \(\constSAT{S}\) is polynomial-time solvable if and only if one of the conditions~\labelcref{item:dichotomy-1} to~\labelcref{item:dichotomy-6} is satisfied.
	Otherwise, \(\constSAT{S}\) is \NP-complete.
	\begin{enumerate}[label=(\arabic*)]
		\item\label{item:dichotomy-1} The constant function \(f_0(x)=0\) is a polymorphism of every relation in \(S\).
		\item\label{item:dichotomy-2} The constant function \(f_1(x)=1\) is a polymorphism of every relation in \(S\).
		\item\label{item:dichotomy-3} The boolean AND function \(f_\land(x,y)=xy\) is a polymorphism of every relation in \(S\).
		\item\label{item:dichotomy-4} The boolean OR function \(f_\lor(x,y)=\max\{x,y\}\) is a polymorphism of every relation in \(S\).
		\item\label{item:dichotomy-5} The function \(\majority(x,y,z)=\left\lfloor(x+y+z)/2\right\rfloor\) is a polymorphism of every relation in \(S\).
		\item\label{item:dichotomy-6} The function \(\minority(x,y,z)=x+y+z\text{ mod }2\) is a polymorphism of every relation in~\(S\).\qedhere
	\end{enumerate}
\end{theoremq}

\section{Color Constrained Paths}\label{sec:covering-path}
In this section, we generalize the quasi-Hamiltonian path problem with prescribed endpoints by considering different constraints on the number of required vertices of each color.
Recall that~\(\C(D)\) is the unique partition of \(D\) into maximal independent sets and that it can be computed in polynomial time.
The quasi-Hamiltonian~$(s,t)$-path problem can be phrased as follows:
Given~$D$ and~$s,t\in V(D)$, decide if there exists an~$(s,t)$-path~$P$ such that~${1\leq |V(P)\cap U|\leq|U|}$ for all~$U\in\C(D)$.
By considering different lower and upper bounds, we obtain the following path problem on semicomplete multipartite digraphs:
\begin{center}
	\fbox{\begin{tabular}{ll}
			\multicolumn{2}{l}{\((a,b,\alpha)\)-color-constrained path problem~$\left[(a,b,\alpha)\text{-}\RCP\right]$}\\
			Instance: & SMD~$D$ such that~$|C|=\alpha$ for all~$C\in\C(D)$,~$s,t\in V(D)$\\
			\multirow{1}{*}{Question:} & \multirow{1}{.75\linewidth}{\justifying Does \(D\) contain an~$(s,t)$-path~$P$ such that~$a\leq|V(P)\cap C|\leq b$ for all~$C\in\C(D)$?}\\
	\end{tabular}}
\end{center}
\noindent The requirement that all color classes are of the same size is not a genuine restriction as any SMD with independence number~$\alpha(D)$ can be made to satisfy this condition by padding each color class with vertices that are not reachable from~$s$, e.g.~vertices that only have outgoing arcs.
This problem is hard for almost all choices of \((a,b,\alpha)\):
\rstMtOne

We prove \cref{thm:general-rcp-hardness} in three steps:
First, we discuss the polynomial-time solvable and open cases.
Second, we prove hardness for~$\alpha=3$ in \cref{thm:rcp-hardness}.
Finally, we lift the hardness from~$\alpha=3$ to the cases~$\alpha\geq4$ by induction over~$\alpha$ in \cref{cor:ccp-hardness}.

\subparagraph*{Polynomial-time solvable and open cases.}
Note that~$(0,\alpha,\alpha)\text{-}\RCP$ is simply the~$(s,t)$-path problem restricted to a special class of digraphs, so this case can be solved in polynomial time using any path-finding algorithm, e.g.~breadth-first search.
The case~${(a,b,\alpha)=(1,1,1)}$ corresponds to the \((s,t)\)-Hamiltonian path problem on semicomplete digraphs.
This problem is known to be polynomial-time solvable using the recursive algorithm of~\cite{DBLP:journals/jal/Bang-JensenMT92}.
Similarly, the problem~\((\alpha,\alpha,\alpha)\text{-}\RCP\) is the \((s,t)\)-Hamiltonian path problem on a special class of semicomplete multipartite digraphs.
The Hamiltonian path problem on general semicomplete multipartite digraphs without prescribed start and end vertices has an elegant solution using~1-path-cycle factors~\cite{DBLP:journals/siamdm/Gutin93}, but the complexity of the~$(s,t)$-Hamiltonian path problem on semicomplete multipartite digraphs remains open.
The remaining open cases both satisfy \(\alpha=2\).
The case~\({(a,b,\alpha)=(1,2,2)}\) corresponds to the quasi-Hamiltonian path problem with prescribed end vertices on the class of semicomplete multipartite digraphs with independence number at most two.
In the case~\((a,b,\alpha)=(1,1,2)\), this is further restricted so that each color class must be hit exactly once.

\subparagraph*{Hardness.}
We prove \NP-completeness of all remaining pairs~$0\leq a\leq b\leq 3$ for~$\alpha=3$ by a reduction from variants of \trSAT.
Given a Boolean formula in conjunctive normal form, an assignment of its variables is called~\emph{$(a,b)$-satisfying} if it satisfies at least~$a$ and at most~$b$ literals in every clause.
Consider the following variants of~\trSAT, which we will call~$(a,b)$-3-$\SAT$:
Given a Boolean formula~$\varphi$ in~3-CNF, decide if~$\varphi$ has an~$(a,b)$-satisfying assignment.
For~$0\leq a\leq 2$ and~$a\leq b\leq 3$, \(\NP\)-completeness of these problems follows from Schaefer's dichotomy theorem, see \cref{prop:dichotomy_example}.
We reduce~$(a,b)$-3-$\SAT$ to~$(a,b,3)\text{-}\RCP$ to prove \NP-hardness of the problem.

Our reduction requires that the same literal does not appear more than once in any given clause, so we prove two preprocessing lemmas that avoid this situation.
\Cref{lem:variable-reduction1} covers the cases where~$1\leq b\leq 2$ and \cref{lem:variable-reduction2} treats the case~$(a,b)=(2,3)$.

\begin{proposition}\label{prop:dichotomy_example}
	For~$0\leq a\leq 2$ and~$\max\{a,1\}\leq b\leq 3$, the problem $(a,b)$-3-$\SAT$ is \NP-complete.
\end{proposition}
\begin{proof}[Proof sketch]
	We only prove hardness of $(0,2)$-3-$\SAT$ to give an example of the application of \cref{thm:dichotomy}.
	The other cases are proved similarly.
	The problem $(0,2)$-3-$\SAT$ is equivalent to~\(\constSAT{S}\) where \(S=\{R_0,R_1,R_2,R_3\}\) and
	\begin{align*}
		R_0=\{(x,y,z)\in\{0,1\}^3: x+y+z\leq 2\} \hspace{12mm}& R_1=\{(x,y,z)\in\{0,1\}^3: x+y-z\leq 1\}\\
		R_2=\{(x,y,z)\in\{0,1\}^3: x-y-z\leq 0\} \hspace{12mm}& R_3=\{(x,y,z)\in\{0,1\}^3: x+y+z\geq 1\}.
	\end{align*}
	The relation \(R_i\) corresponds to a clause where exactly \(i\) variables are negated.
	In the following, functions are applied coordinate-wise as in the definition of polymorphisms.
	As \((0,0,0)\notin R_3\) and~\((1,1,1)\notin R_0\), conditions~\labelcref{item:dichotomy-1,item:dichotomy-2} of \cref{thm:dichotomy} are violated.
	As \({r_1=(1,0,0)\in R_3}\) and \(r_2=(0,1,0)\in R_3\), but \({f_\land(r_1,r_2)=(0,0,0)\notin R_3}\), condition~\labelcref{item:dichotomy-3} is violated.
	As~\({r_3=(1,1,0)}\) and~\({r_4=(0,1,1)}\) are in \(R_0\), but \({f_\lor(r_3,r_4)=(1,1,1)\notin R_0}\), condition~\labelcref{item:dichotomy-4} is violated.
	As the tuples \({r_5=(1,0,0)}\),~\({r_6=(0,1,0)}\) and \(r_7=(0,0,1)\) are in \(R_3\), but~\({\majority(r_5,r_6,r_7)=(0,0,0)}\) is not, condition~\labelcref{item:dichotomy-5} is violated.
	Finally,~\({r_8=(0,1,1)}\),~\({r_9=(1,0,1)}\) and~\({r_{10}=(1,1,0)}\) are in~\(R_0\), but \(\minority(r_8,r_9,r_{10})=(0,0,0)\) is not, so condition~\labelcref{item:dichotomy-6} is violated.
	By \cref{thm:dichotomy}, the problem $(0,2)$-3-$\SAT$ is \(\NP\)-complete.
\end{proof}

\begin{lemma}\label{lem:variable-reduction1}
	Let~$0\leq a\leq b\leq2$, \(b\geq1\) and let~$\varphi$ be a formula in~3-CNF.
	There exists an algorithm which, in polynomial time, determines if~$\varphi$ is~$(a,b)$-satisfiable or computes a CNF-formula~$\psi$ such that each clause of~$\psi$ consists of~2 or~3 pairwise different literals and~$\psi$ is~$(a,b)$-satisfiable if and only if~$\varphi$ is~$(a,b)$-satisfiable.
\end{lemma}
\begin{proof}
	\begin{figure}[!t]
		\centering
		\renewcommand{\arraystretch}{1.03}
		\begin{tabular}{c|cc|cc}
			Clause & \multicolumn{2}{c|}{Output or Operation ($b=1$)} & \multicolumn{2}{c}{Output or Operation ($b=2$)}\\\hline
			\multirow{2}{*}{$A\lor A\lor A$} & & \multicolumn{1}{c}{$a=0$} & $a>0$\\
			& & \multicolumn{1}{c}{$A/0$} & unsatisfiable & \\\hline
			\multirow{2}{*}{$A\lor A\lor B$} & $a=0$ & $a=1$ & $a<2$ & $a=2$\\
			& $A/0$ & $A/0$ and $B/1$ & $A/\overline{B}$ & $A/1$ and $B/0$\\\hline
			$A\lor\overline{A}\lor X$ & \multicolumn{4}{c}{replace by $1\lor X$}\\\hline
			$0\lor X\lor Y$ & \multicolumn{4}{c}{replace by $X\lor Y$}\\\hline
			\multirow{2}{*}{$1\lor A\lor A$} & & \multicolumn{1}{c}{$a<2$} & $a=2$\\
			& & \multicolumn{1}{c}{$A/0$} & unsatisfiable & \\\hline
			$1\lor X\lor Y$ & \multicolumn{2}{c|}{$X/0$ and $Y/0$} & \multicolumn{2}{c}{$X/\overline{Y}$}\\\hline
			\multirow{2}{*}{$A\lor A$} & \multicolumn{2}{c|}{\multirow{2}{*}{$A/0$}} & $a=0$ & $a>0$\\
			& & & remove & $A/1$\\\hline
			$A\lor \overline{A}$ & \multicolumn{4}{c}{replace by 1
			}\\\hline
			$0\lor X$ & \multicolumn{4}{c}{replace by $X$}\\\hline
			\multirow{2}{*}{$1\lor X$} & \multicolumn{2}{c|}{\multirow{2}{*}{$X/0$}} & $a<2$ & $a=2$\\
			& & & remove & $X/1$\\\hline
			\multirow{2}{*}{$X$} & $a=0$ & \multicolumn{2}{c}{$a=1$} & $a=2$\\
			& remove & \multicolumn{2}{c}{$X/1$} & unsatisfiable\\
		\end{tabular}
		\caption{Table of the reduction rules used in \cref{lem:variable-reduction1} to simplify Boolean formulas}
		\label{fig:variable-reduction-1}
	\end{figure}
	The algorithm proceeds by iteratively replacing literals or clauses.
	In each step, the algorithm applies one of the rules displayed in \cref{fig:variable-reduction-1} to simplify the current formula, replacing a clause with a smaller one or removing it entirely while assigning values to variables.
	Here,~$A$ and~$B$ denote different literals,~$C$ denotes a constant and~$X$ and~$Y$ are arbitrary literals or constants.
	The rules are only given up to permutation of their elements, i.e.~$A\lor\overline{A}\lor A$ is treated like~$A\lor A\lor \overline{A}$.
	The operation~$A/X$ replaces the literal~$A$ by~$X$ and the negated literal~$\overline{A}$ by~$\overline{X}$.
	The operation~$C/A$ is defined to be~$A/C$.
	Both~$1/1$ and~$0/0$ induce no operation and the operations~$1/0$ and~$0/1$ cause the algorithm to output unsatisfiable.
	Similarly, the operation~$A/\overline{A}$ causes the algorithm to output unsatisfiable.
	
	The correctness of the presented rules is immediate.
	The first five rules exhaustively treat clauses of length three.
	The remaining rules treat the clause with length at most two.
	In each step, a clause is removed or a literal is removed from a clause, so the algorithm runs in polynomial time with respect to~$|\varphi|$.
	If the resulting formula is empty, the algorithm outputs satisfiable.
	Otherwise, it outputs the resulting formula.
\end{proof}

\begin{lemma}\label{lem:variable-reduction2}
	Let~$\varphi$ be a formula in~3-CNF.
	There exists an algorithm which, in polynomial time, determines if~$\varphi$ is~$(2,3)$-satisfiable or computes a CNF-formula~$\psi$ such that each clause of~$\psi$ consists of~3 pairwise different literals and~$\psi$ is~$(2,3)$-satisfiable if and only if~$\varphi$ is~$(2,3)$-satisfiable.
\end{lemma}
\begin{proof}
	\begin{figure}[!t]
		\centering
		\begin{tabular}{c|c}
			Clause & Output or Operation\\\hline
			~$A\lor A\lor X$ &~$A/1$\\
			~$A\lor\overline{A}\lor X$ &~$X/1$\\
			~$0\lor X\lor Y$ &~$X/1$ and~$Y/1$\\
			~$1\lor 1\lor X$ & remove\\
		\end{tabular}
		\caption{Table of the reduction rules used in \cref{lem:variable-reduction2} to simplify Boolean formulas}
		\label{fig:variable-reduction-2}
	\end{figure}
	\Cref{fig:variable-reduction-2} gives a list of reduction rules that are the first phase of the algorithm.
	The notation is as in the proof of \cref{lem:variable-reduction1}.
	If the resulting formula is empty, the algorithm outputs satisfiable.
	Otherwise, if it has not already been found that~$\varphi$ is not~$(2,3)$-satisfiable, the resulting formula contains the constant~1 at most once in every clause.
	The algorithm then replaces every occurrence of the constant~1 with a new variable~$Z$ to obtain \(\psi\).
\end{proof}

We are now ready to prove hardness of the cases where~$\alpha=3$.

\begin{theorem}\label{thm:rcp-hardness}
	Let \((a,b)\in\{0,1,2,3\}^2\setminus\{(0,0),(0,3),(3,3)\}\).
	Then \((a,b,3)\text{-}\RCP\) is \NP-complete.
\end{theorem}
\begin{proof}
	The witness-characterization of~\(\NP\) implies membership of~$(a,b,\alpha)\text{-}\RCP$ in~$\NP$.
	The reductions presented here are based on the constructions from~\cite{DBLP:journals/dam/Bang-JensenMS13,DBLP:journals/dam/Bang-JensenWY25}.
	A notable difference is that the instances constructed here satisfy~$|C|=3$ for all~$C\in\C(D)$ whereas the instances of~\cite{DBLP:journals/dam/Bang-JensenMS13,DBLP:journals/dam/Bang-JensenWY25} have color classes of arbitrary size.
	Our constructed graphs will only satisfy~$\alpha(D)\leq3$ at first, so there may initially be smaller color classes.
	However, this is not a problem: by iteratively adding vertices which are not reachable from~$s$, i.e.~vertices with no incoming arcs, all color classes can be padded to size~3.
	Denote the result of padding the digraph~$D$ by~$D^+$.
	
	The key elements of the reductions are the B-gadget~\cite{DBLP:journals/dam/Bang-JensenWY25} and the W-gadget~\cite{DBLP:journals/dam/Bang-JensenMS13}, which correspond to the cases~$b\leq 2$ and~$b=3$, respectively.
	These gadgets have the property that any~$(s,t)$-path can contain at most one of two subpaths of a certain kind while satisfying the size constraints of~$(a,b,3)\text{-}\RCP$.
	This allows the modeling of truth assignments by identifying one path with 0 and the other with 1 while introducing one gadget for each variable.
	First, fix~\((a,b)\) with \(0\leq a\leq b\leq2\) and \(b\geq1\).
	
	The B-gadget~$B[s,t,n_1,n_2]$ has vertex set~$V=\{s,t,y_0,\dots,y_{n_1},z_1,\dots,z_{n_2+1}\}$ and the following arcs:~$sy_0\dots y_{n_1}t$ and~$sz_1\dots z_{n_2+1}t$ are paths such that every vertex on the path dominates all vertices preceding it except for its direct predecessor and~$\{{y_0,\dots, y_{n_1}\}\rightarrow\{z_1,\dots,z_{n_2+1}}\}$.
	The vertices of the gadgets will be partitioned differently depending on~$(a,b)$, so some of the arcs will be removed later.
	As an example, the B-gadget \(B[s,t,1,2]\) is visualized in \cref{fig:bgadget}.
	
	\begin{figure}
		\centering
		\includegraphics[scale=0.8]{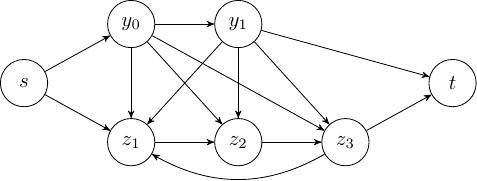}
		\caption{The~$B[s,t,1,2]$-gadget as defined in~\cite{DBLP:journals/dam/Bang-JensenWY25}.
			For visual clarity, the incoming arcs of \(s\) and the outgoing arcs of \(t\) are not drawn.}
		\label{fig:bgadget}
	\end{figure}
	
	We now specify the reduction.
	Let~$\varphi$ be a formula in~3-CNF.
	Apply \cref{lem:variable-reduction1} to determine if~$\varphi$ is~$(a,b)$-satisfiable or obtain the formula~$\psi$.
	In the first case, map~$\varphi$ to some fixed yes-instance or no-instance depending on whether it is satisfiable or not.
	Otherwise, let~$X_1,\dots,X_n$ be the variables of~$\psi$ and~$C_1,\dots,C_m$ be the clauses of~$\psi$.
	Let~$X_i$ occur~$p_i$ times as a literal and~$q_i$ times as a negated literal.
	For all~$i$ such that~$1\leq i\leq n$, we associate with~$X_i$ the B-gadget~$B^i\coloneqq B[s^i,t^i,p_i,q_i]$ and give each vertex in this gadget the superscript~$i$ to uniquely identify all vertices.
	The~$j$-th occurrence of~$X_i$ as a positive (resp. negative) literal corresponds to~$y_j^i$ (resp.~$z_j^i$).
	Denote by~$V_i$ the set of vertices associated with the literals of the clause~$C_i$.
	Identify~$t^i$ with~$s^{i+1}$ and add arcs such that~$B^j\Rightarrow B^i$ for all~$i$ and~$j$ such that~$i<j$.
	If~$b=2$, also add the vertices~$w_1,\dots,w_{2n+1}$ such that~$N^+(w_i)=\{w_{i+1}\}$ for all~$i$ such that~$1\leq i<2n+1$ and~$N^+(w_{2n+1})=\{s_1\}$.
	All vertices not in~$N^+(w_i)$ dominate~$w_i$.
	The color classes are defined as follows:
	\begin{itemize}
		\item~$V_i$ is a color class for all~$1\leq i\leq m$.
		\item~$\{y_0^i,z_{q_i+1}^i,w_i\}$ is a color class for all~$1\leq i\leq n$ (if~$b\neq2$, remove \(w_i\) from the set).
		\item~$\{s_i,w_{n+i}\}$ is a color class for all~$1\leq i\leq n$.
		The same remark as above applies.
		\item~$\{t_n,w_{2n+1}\}$ is a color class.
	\end{itemize}
	Remove all arcs between vertices in the same color class.
	Call the resulting SMD~$D$.
	Set~$s\coloneqq w_1$ if~$b=2$ and~$s\coloneqq s_1$ otherwise.
	Set~$t\coloneqq t_n$.
	We map~$\varphi$ to~$(D^+,s,t)$.
	We now prove correctness of our reduction.
	
	``$\Rightarrow$'': Let~$P$ be an~$(s,t)$-path witnessing that~$(D^+,s,t)$ is a yes-instance.
	The vertex~$t^i$ separates~$B^i$ from~$B^{i+1}$, so~$B^1,\dots,B^n$ are traversed in this order in \(P\).
	In particular, \(P\) uses all of the vertices~$s_1,\dots,s_n$.
	As \(P\) satisfies the upper bound~$b$, it can only contain one element of~$\{y^i_0,z^i_{q_i+1}\}$ for any \(1\leq i\leq n\).
	If~$b=1$, this is immediate as these vertices are in the same color class.
	If~$b=2$, \(P\) starts with the subpath~$w_1\dots w_{2n+1}s_1$ and~$w_i$ is in the same color class as~$y^i_0$ and~$z^i_{q_i+1}$.
	Therefore, the size constraint of~2 for~$\{y^i_0,z^i_{q_i+1},w_i\}$ is effectively a size constraint of~1 for~$\{y^i_0,z^i_{q_i+1}\}$.
	
	This implies that \(P\) contains either the subpath~$s^iy^i_0\dots y^i_{p_i}t^i$ or~$s^iz^i_1\dots z^i_{q_i+1}t^i$.
	The vertex~$s^i$ has only two positive neighbors,~$y^i_0$ and~$z^i_1$.
	If~$P$ contains~$z^i_j$, then~$z^i_j\dots z^i_{q_i+1}t^i$ is a subpath of~$P$ as~${N^+(z^i_j)\subseteq\{z^i_1,\dots,z^i_{j+1}\}}$ for all~$1\leq j\leq q_i$.
	This makes using an arc to a previous~$z$-vertex impossible when trying to reach~$t^i$.
	Additionally,~$P$ cannot contain~$y^i_0$ and~$z^i_j$ for any~$j$ such that~$1\leq j\leq q_i+1$ as~$y^i_0$ and~$z^i_{q_i+1}$ are mutually exclusive.
	Therefore,~$P$ must contain the subpath~$s^iy_0^i\dots y^i_{p_i}t^i$ if it does not contain~$s^iz^i_1\dots z^i_{q_i+1}t^i$ as this is the only path from~$s^i$ to~$t^i$ that does not use any of the~$z^i_j$.
	
	We construct an \((a,b)\)-satisfying assignment \(\sigma:\{X_1,\dots,X_n\}\to\{0,1\}\) in the following way:
	Set~${\sigma(X_i)\coloneqq1}$ if~$s^iy_0\dots y^i_{p_i}t^i$ is a subpath of~$P$, else set~$\sigma(X_i)\coloneqq0$.
	As the literals associated to the clause~$C_i$ form a color class in~$D^+$, this assignment sets at least~$a$ and at most~$b$ literals to~$1$ for each clause.
	
	``$\Leftarrow$'': Let $\sigma$ be an \((a,b)\)-satisfying assignment of \(\psi\).
	Let~$P_0=w_1\dots w_{2n+1}s_1$ if~$b=2$ and~$P_0=s_1$ if~$b\neq2$.
	For each~$1\leq i\leq n$, set~$P_i\coloneqq y_0\dots y^i_{p_i}t^i$ if~$\sigma(X_i)=1$ and~$P_i\coloneqq z^i_1\dots z^i_{q_i+1}t^i$ if~$\sigma(X_i)=0$.
	These paths exist in \(D^+\) as the literals in each clause are pairwise different so that none of the arcs of the paths are removed.
	Then~$P_0\dots P_n$ is an~$(s,t)$-path satisfying the cardinality constraints imposed by the parameters~$(a,b)$.
	
	This leaves the case~$b=3$ and~$a\in\{1,2\}$, both of which use the \(W\)-gadget.
	The prototypical~$W$-gadget is visualized in \cref{fig:wgadget}.
	The~$W$-gadget~$W[x,y,n_1,n_2]$ is then obtained from \cref{fig:wgadget} by replacing~$s$ with~$x$,~$t$ with~$y$,~$T$ with a clique~$y_1,\dots,y_{n_1+1}$ and~$F$ with a clique~$z_1,\dots,z_{n_2+1}$.
	Here, a clique is a set of vertices where all possible arcs between any pair of different vertices exist.
	The arcs between the cliques and the remaining vertices are given by the arcs of \cref{fig:wgadget}.
	
	\begin{figure}
		\centering
		\includegraphics[scale=0.8]{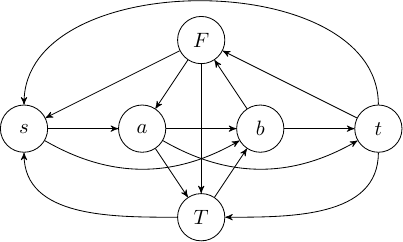}
		\caption{The~$W$-gadget as defined in~\cite{DBLP:journals/dam/Bang-JensenMS13}.}
		\label{fig:wgadget}
	\end{figure}
	
	Let~$\varphi$ be an instance of~$(a,b)$-3-$\SAT$.
	If~$a=1$, set~$\psi\coloneqq\varphi$.
	If~$a=2$, apply \cref{lem:variable-reduction2} to determine if~$\varphi$ is~$(2,3)$-satisfiable or obtain the formula~$\psi$.
	In the first case, map~$\varphi$ to some fixed yes-instance or no-instance depending on whether it is satisfiable or not.
	Let~$X_1,\dots,X_n$ and~$C_1,\dots,C_m$ denote the variables and clauses of~$\psi$.
	Let~$p_i$ be the number of positive occurrences of~$X_i$ and~$q_i$ the number of its negative occurrences.
	Introduce a~$W$-gadget~$W^i[s^i,t^i,p_i,q_i]$ for each~$1\leq i\leq n$.
	Give each vertex in~$W^i$ the superscript~$i$ to uniquely identify it.
	The~$j$-th positive (resp. negative) occurrence of~$X_i$ is associated with~$y^i_j$ (resp.~$z^i_j$).
	The vertices~\(y^i_{p_i+1}\) and~\(z^i_{q_i+1}\) are not associated to any literal.
	They ensure that the paths~\(s^ia^iy^i_1\dots y^i_{p_i+1}b^it^i\) and~\(s^ib^iz^i_1\dots z^i_{p_i+1}a^it^i\) exist even if \(p_i=0\) or \(q_i=0\).
	Denote the vertices associated to the literals of clause~$C_i$ by~$V_i$.
	Add arcs such that~$W_j\Rightarrow W_i$ and add the arc~$t_is_{i+1}$ for all~$i$ and~$j$ such that~$1\leq i<j\leq n$.
	The color classes are~$\{s^i,t^i\}$,~$\{a^i,b^i\}$,~$\{y^i_{q_i+1},z^i_{p_i+1}\}$ and~$V_i$ for all~$i$ such that~$1\leq i\leq n$.
	Remove arcs between vertices of the same color and call the resulting SMD~$D$.
	We map~$\varphi$ to~$(D^+,s^1,t^n)$.
	
	``$\Rightarrow$'': Let~$P$ be an~$(s,t)$-path witnessing that~$(D^+,s^1,t^n)$ is a yes-instance.
	The crucial property of the~$W$-gadget is that any~$(s,t)$-path cannot contain both~$F$ and~$T$.
	This was proven in~\cite[Theorem~3.1]{DBLP:journals/dam/Bang-JensenMS13} but can also be seen by inspecting \cref{fig:wgadget}.
	Replacing~$F$ and~$T$ with cliques does not change this fact.
	Set~$\sigma(X_i)\coloneq1$ if~$\{y^i_1,\dots,y^i_{p_i}\}\cap V(P)\neq\emptyset$ and~$\sigma(X_i)=0$ otherwise.
	As~$P$ contains at least~$a$ vertices from each~$V_i$, this assignment satisfies at least~$a$ variables from each clause, so~$\varphi$ is~$(a,b)$-satisfiable.
	
	``$\Leftarrow$'': Let~$\sigma$ be an assignment witnessing that~$\varphi$ is~$(a,b)$-satisfiable.
	Let~$Q_i$ be a quasi-Hamiltonian path in~$D^+[y^i_1,\dots,y^i_{p_i}]$ if~$\sigma(X_i)=1$ and a quasi-Hamiltonian path in~$D^+[z^i_1,\dots,z^i_{q_i}]$ if~$\sigma(X_i)=0$.
	Set~${P_i=s^ia^iQ_ib^it^i}$ if~$\sigma(X_i)=1$ and~$P_i=s^ib^iQ_ia^it^i$ otherwise.
	Then~$P=P_1\dots P_n$ witnesses that~$(D^+,s^1,t^n)$ is a yes-instance as~$\sigma$ satisfies at least~$a$ variables in every clause.
	In the case \(a=2\), the literals in each clause are pairwise different, so that each color class \(V_i\) will be hit twice.
\end{proof}

\begin{corollary}\label{cor:ccp-hardness}
	Let~$(a,b,\alpha)\in\Z^3$ such that~$\alpha\geq3$,~$(a,b)\notin\{(0,0),(0,\alpha),{(\alpha,\alpha)}\}$ and~$0\leq a\leq b\leq \alpha$.
	Then~$(a,b,\alpha)\text{-}\RCP$ is \NP-complete.
\end{corollary}
\begin{proof}
	The witness-characterization of~\(\NP\) implies membership of~$(a,b,\alpha)\text{-}\RCP$ in~$\NP$.
	The hardness for~$\alpha=3$ follows from \cref{thm:rcp-hardness}.
	We inductively prove hardness for~$\alpha+1>3$ assuming the statement for~$\alpha$.
	By padding all color classes with unreachable vertices,~$(a,b,\alpha)\text{-}\RCP$ reduces to~$(a,b,\alpha+1)\text{-}\RCP$.
	Additionally, changing the start vertex to a new vertex~$s'$ and forcing every~$(s',t)$-path to begin with a path covering one vertex from each color class whose final vertex can then only go to~$s$ shows that~$(a,b,\alpha)\text{-}\RCP$ reduces to~$(a+1,b+1,\alpha+1)\text{-}\RCP$.
	This is similar to how paths are prefixed in the first part of the proof of \cref{thm:rcp-hardness}.
	
	This leaves the cases~$(0,\alpha,\alpha+1)$ and~$(1,\alpha+1,\alpha+1)$.
	In the first case, the starting vertex of an instance of \((0,\alpha-1,\alpha)\text{-}\RCP\) may be modified such that one new vertex from each color class is hit before reaching the original start vertex to reduce \((0,\alpha-1,\alpha)\text{-}\RCP\) to \((0,\alpha,\alpha+1)\text{-}\RCP\).
	In the second case, the color classes of an instance of \((1,\alpha,\alpha)\text{-}\RCP\) may simply be padded with unreachable vertices to reduce \((1,\alpha,\alpha)\text{-}\RCP\) to \((1,\alpha+1,\alpha+1)\text{-}\RCP\).
\end{proof}

Our analysis shows that~$(a,b,\alpha)\text{-}\RCP$ is \NP-complete for almost all choices of~$a$ and~$b$ given that~$\alpha$ is at least three while particular choices are polynomial-time solvable.
This establishes a partial classification that mirrors Schaefer's dichotomy theorem.
However, we did not establish a complete dichotomy as the complexity of~$(1,1,2)\text{-}\RCP$,~$(1,2,2)\text{-}\RCP$ and~$(\alpha,\alpha,\alpha)\text{-}\RCP$ remains open as an interesting problem for future work.
The complexity of the corresponding~$\SAT$-problems suggests that these problems may be solvable in polynomial time.

\section{Structural Results}\label{sec:structural-results}
It is noteworthy that the case~$\alpha=2$ seems to play a special role for the color-constrained path problem.
This is similar to how 2 is a ``magic number'' for the chromatic number, e.g.~Yeo~\cite{DBLP:journals/jal/Yeo99} obtained stronger results for semicomplete bipartite digraphs than for general semicomplete multipartite digraphs.
This motivates us to study the class of semicomplete multipartite digraphs with independence number at most 2 in this section.
We are unaware of any previous work that expressly studies this particular class.
We prove sufficient conditions for quasi-Hamiltonian paths as generalizations of results for semicomplete digraphs in \cref{thm:not-2-strong-without-xy,thm:2-sep-all-trivial} and a simpler characterization of Hamiltonian paths in this graph class in \cref{thm:hamiltonian-criterion-nsmd}.

\subsection{Sufficient Conditions for Quasi-Hamiltonian Paths}
The algorithm for detecting Hamiltonian~$(s,t)$-paths in semicomplete digraphs from~\cite{DBLP:journals/jal/Bang-JensenMT92} is based on three sufficient conditions for Hamiltonian \((s,t)\)-paths in semicomplete digraphs.
The first is due to Thomassen~\cite[Theorem 5.4]{DBLP:journals/jct/Thomassen80a} and corresponds to \cref{thm:3-paths} restricted to semicomplete digraphs.
Bang-Jensen, Manoussakis and Thomassen~\cite[Theorem~3.3 and~3.4]{DBLP:journals/jal/Bang-JensenMT92} proved the other two.
We augment the results of~\cite{DBLP:journals/dam/GuoLS12} by also lifting these two theorems to sufficient conditions for quasi-Hamiltonian paths in semicomplete multipartite digraphs.
The first condition is lifted to general semicomplete multipartite digraphs in~\cref{thm:not-2-strong-without-xy}, while we lift the second condition to semicomplete multipartite digraphs with independence number at most 2 in~\cref{thm:2-sep-all-trivial}.
Before proving the main results of this section, we isolate a simple argument that appears multiple times to avoid repetitions and then prove a minor result needed for the proof of \cref{thm:2-sep-all-trivial}.

\begin{proposition}\label{prop:shortcut}
	Let~$D=(V,A)$ be a 2-strong semicomplete multipartite digraph with distinct vertices~$x,y,z\in V$ such that there exists an~$(x,y)$-quasi-Hamiltonian path~$P=v_0v_1\dots v_{\ell+1}$ in~$D-z$.
	If any of the following conditions is satisfied, then~$D$ contains an~$(x,y)$-quasi-Hamiltonian path:
	\begin{enumerate}[label=(\arabic*)]
		\item\label{prop:shortcut-1}~$|H(z)|>1$
		\item\label{prop:shortcut-2} There exist~$i$ and~$j$ with~$0\leq i<j\leq \ell+1$ such that~$v_iz,zv_j\in A$.
		\item\label{prop:shortcut-3}~$xz\in A$ or~$zy\in A$
	\end{enumerate}
\end{proposition}
\begin{proof}
	\noindent\labelcref{prop:shortcut-1}: In this situation,~$P$ is also an~$(x,y)$-quasi-Hamiltonian path in~$D$.
	
	\noindent\labelcref{prop:shortcut-2}: The path~$P$ and~$v_0\dots v_izv_j\dots v_{\ell+1}$ can be merged using \cref{lem:merging} to obtain the desired path.
		 
	\noindent\labelcref{prop:shortcut-3}: As we may assume that neither of the previous situations are applicable,~$|H(z)|=1$.
	Consider the case~$zy\in A$.
	As~$D$ is 2-strong, there exists a path from~$V(P)$ to~$z$ that does not contain~$y$.
	Let~$Q$ be a shortest such path and let its starting vertex be~$v_i$.
	Then~$P$ and~$v_0\dots v_{i-1}Qy$ can be merged using \cref{lem:merging} to obtain the desired path.
	The case~$xz\in A$ is treated analogously by considering a shortest path from~$z$ to~$V(P)$ that does not contain~$x$.
\end{proof}

\begin{lemma}[Generalization of {\cite[Lemma~3.2]{DBLP:journals/jal/Bang-JensenMT92}}]\label{lem:2-sep-arcs}
	Let~$D=(V,A)$ be a semicomplete multipartite digraph and~$x,y\in V$ such that there exist two internally disjoint~$(x,y)$-paths and a~2-separator~$\{u, v\}$ of~$x$ and~$y$ in~$D$.
	Let~$D'$ denote the digraph obtained from~$D$ by adding the arcs~$vu$ and~$uv$ if they do not exist already.
	Then~$D$ contains an~$(x,y)$-quasi-Hamiltonian path if and only if~$D'$ contains an~$(x,y)$-path intersecting each element of~$\C(D)$.
\end{lemma}
\begin{proof}
	Clearly, \(D'\) contains an~$(x,y)$-path intersecting each element of~$\C(D)$ if~$D$ contains an~$(x,y)$-quasi-Hamiltonian path.
	Let~$P$ be an~$(x,y)$-path intersecting each element of~$\C(D)$ in~$D'$.
	If~$P$ does not use any of the additional arcs, there is nothing to prove, so assume~$P$ uses the new arc~$vu$.
	The case where~$P$ uses~$uv$ instead is symmetric.
	Let \(P_1\) and \(P_2\) be two internally disjoint~$(x,y)$-paths in \(D\).
	Because \(\{u,v\}\) is a separator of \(x\) and \(y\), one of the paths contains \(u\) and the other path contains \(v\).
	Without loss of generality, assume that~\(P_1\) contain \(v\).
	Let \(w_1\) be the first vertex in \(P_1[v,y]\) from \(V(P[u,y])\) and let \(w_2\) be the last vertex in \(P_2[x,u]\) from~\(V(P[x,v])\).
	Then \(V(P[x,v[)\) is disjoint from \(V(P_1]v,w_1])\) as otherwise~\(\{u,v\}\) would not separate \(x\) and \(y\).
	Similarly, \(V(P]u,y])\) is disjoint from \(V(P_2[w_2,u[)\).
	Define~${Q_1\coloneq P[xv]P_1]v,w_1[P[w_1,y]}$ and~${Q_2\coloneq P[xw_2]P_2]w_2,u[P[u,y]}$.
	These paths only use arcs that are present in~$D$ and~$Q_1\cup Q_2$ does not contain a directed cycle.
	Applying \cref{lem:merging} yields an~$(x,y)$-quasi-Hamiltonian path in~$D$.
\end{proof}

Care must be taken when applying \cref{lem:2-sep-arcs}.
While~\cite[Lemma~3.2]{DBLP:journals/jal/Bang-JensenMT92}, i.e.~\cref{lem:2-sep-arcs} restricted to semicomplete digraphs, transforms one semicomplete digraph into another semicomplete digraph, applying \cref{lem:2-sep-arcs} to a semicomplete multipartite digraph is only guaranteed to yield another semicomplete multipartite digraph when the independence number is at most 2.
This is caused by the fact that when~$u$ and~$v$ are part of an independent set of size~3, say~$\{u,v,w\}$, adding arcs between~$u$ and~$v$ forces the existence of at least one arc between~$\{u,v\}$ and~$w$ in order to satisfy the definition of semicomplete multipartite digraphs.

\rstMtTwo
\begin{proof}
	This is the most technical proof of the paper, so we first give a high-level overview before beginning the proof itself.
	Using the structure of the linear decomposition \(R_1,\dots,R_k\) of \(D-\{x,y,z\}\), we obtain an \((x,y)\)-quasi-Hamiltonian path \(P=v_0\dots v_{\ell+1}\) in \(D-z\) with \cref{lem:non-strong}.
	By \cref{prop:shortcut}, it then only remains to integrate \(z\) into \(P\) to obtain an \((x,y)\)-quasi-Hamiltonian path in \(D\).
	We choose \(P\) as the longest possible quasi-Hamiltonian path with~\(v_1\in R_1\) and \(v_\ell\in R_k\).
	A longest path is used as this restricts the arcs between vertices on \(P\) and vertices outside of \(P\) which simplifies the analysis.
	We first show that \(zv_2,v_{\ell-1}z\in A\) in \cref{claim:second-vertex-arc-to-z}.
	Technically, \cref{claim:second-vertex-arc-to-z} only shows \(v_2z\in A\).
	However, note that reversing all arcs in \(D\) and swapping the roles of \(x\) and \(y\) yields a graph that also satisfies the conditions of \cref{thm:not-2-strong-without-xy}.
	Applying \cref{claim:second-vertex-arc-to-z} to this reversed graphs yields \(v_{\ell-1}z\in A\).
	
	The arcs \(zv_2\) and \(v_{\ell-1}z\) give a starting point for integrating \(z\) into the path \(P\):
	Find a way to skip over \(v_2\), then follow \(P\) to \(v_{\ell-1}\), go to \(v_2\) via \(z\) and find a way to reach \(y\) without using any of the previously visited vertices.
	Vertices from \(R_1\) and \(R_k\) often play a special role in this process.
	Vertices from \(R_1\) dominate all vertices from other linear components offering an opportunity to skip over \(v_2\).
	Similarly, vertices from \(R_k\) are dominated by all vertices from other linear components, so they are potential re-entry points into \(P\) from \(v_2\).
	The vertex \(v_l\) is not the only negative neighbor of \(y\) in \(D-x\).
	We call such vertices dominators.
	Dominators offer an alternative way to reach \(y\) from \(v_2\) that does not require going to \(v_l\).
	This is used when \(v_l\) is already required to skip over \(v_2\).
	
	To prove \cref{claim:second-vertex-arc-to-z}, we show that an \((x,y)\)-quasi-Hamiltonian path exists if \(v_2\rightarrow z\).
	Under this assumption, there exists a positive neighbor \(u\) of \(z\) outside of \(P\) that lies in the same linear component \(R_\alpha\) as \(v_2\).
	In the case \(\alpha<k\), we use the idea of skipping over \(v_2\) and then going to~\(v_l\) from \(v_2\) after moving through \(z\).
	In the case \(\alpha=k\), we use a dominator to get to \(y\) from \(v_2\).
	
	After proving \cref{claim:second-vertex-arc-to-z}, we proceed with a case distinction depending on the index \(q\) of the first vertex of \(P\) in \(R_k\).
	When the vertex \(v_q\) is in the middle of the path, it is easy to integrate~\(z\) into \(P\).
	Only the cases \(q\leq 3\) and \(\ell-1\leq q\) require deeper analysis.
	The main part of the proof deals with the case \(q\leq 3\), which is divided into two subcases.
	In the first subcase, \(P\) does not contain a dominator.
	We show that there either exists an \((x,y)\)-quasi-Hamiltonian path or a vertex of \(P\) can be replaced with a dominator.
	This replacement is carried out in a careful way so as to ensure that the new path also satisfies \(q\leq3\).
	This leads to the second subcase, in which~\(P\) contains a dominator.
	We then show that an \((x,y)\)-quasi-Hamiltonian path exists in this subcase.
	
	The case \(\ell-1\leq q\) is significantly restricted by another reversal argument:
	Unless the structure is very constrained, reversing all arcs results in a digraph that falls into one of the previously analyzed cases resulting in an \((x,y)\)-quasi-Hamiltonian path.
	A simple analysis then completes the proof.
	
	We now begin with the proof.
	We may assume that~$H(x)$ and~$H(y)$ are singletons by adding the arcs~$vx$ or~$yv$ whenever necessary.
	This only makes finding an~$(x,y)$-QHP harder.
	Let~\(R_1,\dots,R_k\) be the linear decomposition of \(D-\{x,y,z\}\).
	As~$D-x$,~$D-y$ and \(D-z\) are strong,~$x$,~$y$ and~$z$ each have a positive neighbor in~$R_1$ as well as a negative neighbor in~$R_k$.
	In particular, there is an~$(x,y)$-quasi-Hamiltonian path in~$D-z$ whose second vertex is in~$R_1$ and whose penultimate vertex is in~$R_k$ by \cref{lem:non-strong}.
	Let~${P=v_0v_1\dots v_\ell v_{\ell+1}}$ be a longest path in~$D-z$ with these properties.
	By \cref{prop:shortcut}, we may assume that~$|H(z)|=1$ and that there exists a~$j$ with~$0<j<\ell+1$ such that~$\{v_{j+1},\dots,v_{\ell+1}\}\rightarrow z\rightarrow\{v_0,\dots,v_{j-1}\}$.
	Because~$D-\{x,y,z\}$ consists of at least 4 color classes, it follows that~$l\geq 4$.
	
	\begin{claim}\label{claim:second-vertex-arc-to-z}
		We may assume~$zv_2\in A$.
	\end{claim}
	\begin{claimproof}[Proof of \cref{claim:second-vertex-arc-to-z}]
		Assume that this is not the case, so~$v_2\rightarrow z$.
		Then we may assume~$v_i\rightarrow z$ for all~$i\geq 2$ by \cref{prop:shortcut}.
		Let~$Q=z\dots v_i$ be a path from~$z$ to~$V(P)$ in~$D-\{x,v_1\}$ that is internally disjoint from~$V(P)$.
		If~${i\neq 2}$ or~${v_1z\in A}$,~$xv_1\dots v_{i-1}Qv_{i+1}\dots v_{\ell+1}$ is the desired path.
		Otherwise, we may assume that~$z\rightarrow v_1$ and~$Q$ can only be chosen such that~$i=2$.
		As~$v_2\rightarrow z$,~$Q$ consists of at least~3 vertices.
		The successor \(u\) of \(z\) in \(Q\) satisfies~${u\notin V(P)}$ and~${zu\in A}$.
		Let~$v_2\in R_\alpha$ and~$u\in R_\beta$.
		By the choice of~$u$,~$\{v_3,\dots,v_{\ell+1}\}\Rightarrow u$ and~$v_2$ is reachable from~$u$.
		Therefore, the structure of the linear decomposition implies~$\alpha=\beta$ as one of~$\{v_3,\dots,v_\ell\}$ must have an arc to~$u$ due to the fact that~$l\geq4$.
		
		Assume~$\alpha<k$.
		Then~$H(u)=H(v_\ell)$ as otherwise~$u\rightarrow v_l$, a contradiction.
		Consequently,~\({v_{\ell-1}u\in A}\), so that \(\{v_2,\dots,v_{\ell-1}\}\subseteq R_\alpha\).
		This also implies~$R_k\subseteq H(u)$ because any vertex in~$R_k\setminus H(u)$ yields the desired path by the structure of the linear decomposition.
		Additionally, we may assume~\(Q=zuv_2\) as the successor of \(u\) on \(Q\) must be adjacent to \(v_\ell\).
		Assume there exists a \((v_1,v_j)\)-path~$Q_j$ in~\({D-\{x,y\}}\) internally disjoint from~$V(P)$ for some~$2<j<l$.
		This implies~$u\notin V(Q_j)$ as otherwise we obtain a path from~\(z\) to \(v_j\) that is internally disjoint from~\({V(P)}\).
		Then one of~$xQ_jv_{j+1}\dots v_{\ell-1}zuv_2\dots v_{j-1}v_{\ell}y$ or~$xQ_j v_{j+1}\dots v_{\ell-1}zuv_2\dots v_{j-2}v_{\ell}y$ is the desired path.
		Thus, we may assume that any path in~$D-\{x,y\}$ from~$v_1$ to another vertex in~$V(P)$ must use~$v_2$ or~$v_\ell$.
		The only positive neighbors of~$v_\ell$ in~$D-y$ are~$x$ and~$z$.
		Consider two internally disjoint~$(z,v_3)$-paths~$P_1,P_2$ in~$D-y$.
		Then one path must use~$x$, say~$P_1$, and it cannot contain~$v_2$ as this would imply that~$z$ can reach a vertex from~$V(P)$ in~$D-y$ without using a vertex from~$\{x,v_1,v_2\}$ through~$P_2$.
		Similarly, if~$v_1\in V(P_1)$, then~$v_1$ must occur on~$P_1$ before~$x$ as otherwise~$P_1$ must also contain~$v_2$.
		Let~$v_\gamma$ be the first vertex in~$P_1$ from~$V(P)$ after~$x$.
		By the previous analysis, \(\gamma\geq3\).
		Then~$P_1[x,v_\gamma]P]v_\gamma,v_{\ell-1}]zv_1\dots v_{\gamma-1}v_{\ell}y$ or~$P_1[x,v_\gamma]P]v_\gamma,v_{\ell-1}]zv_1\dots v_{\gamma-2}v_{\ell}y$ is the desired path.
		
		This leaves only the case~$\alpha=k$.
		Then~$R_1\subseteq H(u)$ and~$Q=zuv_2$ as otherwise we would obtain a longer~$(x,y)$-QHP in~$D-z$.
		Assume that~$|R_1|>1$ and let~$w\in R_1-v_1$.
		It follows that~$zw\in A$ as~$D-x$ is 2-strong and~$N^-(w)\subseteq\{x,y,z\}$.
		Then~$xv_1v_2zwv_3\dots v_{\ell+1}$ or~$xv_1v_2zwv_4\dots v_{\ell+1}$ is the desired path, so we may assume~$|R_1|=1$.
		We will call vertices in~$N^-(y)-\{x,v_\ell\}$ dominators.
		Dominators exist as~$D-x$ is 2-strong.
		It is impossible for~$v_1$ to be a dominator as~$|R_1|=1$ and \(y\) does not lie on a 2-cycle.
		Similarly,~$u$ cannot be a dominator as then~$xv_1\dots v_{\ell}zuy$ is the desired path.
		Assume that~$v_i\in V(P)$ is a dominator.
		One of the paths~$xv_1v_{i+1}\dots v_{\ell}zuv_2\dots v_iy$ or~$xv_1v_{i+2}\dots v_{\ell}zuv_2\dots v_iy$ is the desired path unless~${i=\ell-1}$ and~$H(v_1)=H(v_\ell)$.
		The vertex~$v_\ell$ has a predecessor~$w$ in~$D-x$ that is not~$v_{\ell-1}$.
		It follows that~${H(w)\neq H(v_\ell)=H(v_1)}$, so~$v_1w\in A$.
		If~$w\notin V(P)$,~$xv_1wv_{\ell}zuv_2\dots v_{\ell-1}y$ is the desired path.
		Else,~$w=v_j$ for some~$j$ so that one of~$xv_1v_{j+1}\dots v_{\ell-1}zuv_2\dots v_jv_{\ell}y$ or~$xv_1v_{j+2}\dots v_{\ell-1}zuv_2\dots v_jv_{\ell}y$ is an \((x,y)\)-quasi-Hamiltonian path in~$D$.
		
		Now assume~$V(P)$ does not contain a dominator.
		Let~$Q_1=v_\zeta\dots w$ and~$Q_2=v_\xi\dots w$ be different internally disjoint paths from~$V(P)$ to a dominator~$w$ in~$D-y$.
		Such paths exist as~$D-y$ is 2-strong.
		Assume that~$v_\zeta$ and~$v_\xi$ are the only vertices in~$Q_1$ and~$Q_2$ in~$V(P)$.
		If this is not the case, say there exists~${v_{\zeta'}\in V(Q_1)\cap V(P)}$ with~$\zeta'\neq\zeta$, then set~$Q_1\coloneq Q_1[v_{\zeta'},w]$.
		Iterating this process and an analogous one for~$Q_2$ achieves the condition.
		If~$\zeta\neq \ell-1$, choose~$Q_0\coloneq Q_1$.
		If~$\zeta=\ell-1$ and~${\xi\neq \ell-1}$, choose~$Q_0\coloneq Q_2$.
		Otherwise, choose the longer of the two paths as~$Q_0$ breaking ties by choosing~$Q_1$.
		Let~$Q_0=v_i\dots w$.
		Then~$z\notin V(Q_0)$ as otherwise merging~$P$ and~$P[x,v_i[Q_0y$ yields the desired path.
		Similarly, we obtain~$w\rightarrow z$ as the arc~$zw$ yields the path~$v_0\dots v_{\ell}zwy$.
		It follows that~$u\notin V(Q_0)$ as this yields~$v_0\dots v_{\ell}zP_0[u,w]y$, which is the desired path.
		If~$i=l$, we obtain a longer~$(x,y)$-QHP in~$D-z$.
		If~$i=\ell-1$, then also~$\zeta=\xi=\ell-1$ and so~$Q_0$ is the longer path among~$Q_1$ and~$Q_2$.
		In particular,~$Q_0$ has at least one internal vertex as~$Q_1$ and~$Q_2$ are different, so that~$v_0\dots v_{\ell-2}Q_0y$ is a longer~$(x,y)$-QHP in~$D-z$.
		If~$i\leq 1$,~$P_0zv_1\dots v_{\ell}y$ or~$xP_0zuv_2\dots v_{\ell}y$ is the desired path.
		This implies~$1<i<\ell-1$, so one of~$xv_1v_{i+1}\dots v_{\ell}zuP[v_2,v_i[P_0y$ or~$xv_1v_{i+2}\dots v_{\ell}zuP[v_2,v_i[P_0y$ is the desired path.
	\end{claimproof}
	
	By reversing all arcs in \(D\), we obtain a digraph \(D'\) which also satisfies the conditions of \cref{thm:not-2-strong-without-xy}.
	Applying \cref{claim:second-vertex-arc-to-z} to \(D'\) yields \(zv_{\ell-1}\in A(D')\), i.e.~\(v_{\ell-1}z\in A\).
	The existence of the arcs \(zv_2\) and \(v_{\ell-1}z\in A\) implies~$zv_1,v_{\ell}z\in A$ by \cref{prop:shortcut}.
	Let~$q$ be the minimal index such that \(v_q\) is in \(R_k\).
	Then~$\{v_1,\dots,v_{q-1}\}\Rightarrow R_k$.
	The vertices~$x$ and~$y$ are not a part of the linear decomposition and~$v_1$ lies in~$R_1$, so~$2\leq q\leq l$.
	
	\textbf{Case 1},~$3<q<\ell-1$:
	If~$v_1v_q\in A$, set~$a=q$, otherwise set~$a=q+1$.
	If~$v_{q-1}v_\ell\in A$, set~$b=q-1$, otherwise set~$b=q-2$.
	Then~$xv_1v_av_{a+1}\dots v_{\ell-1}zv_2v_3\dots v_bv_{\ell}y$ is an~$(x,y)$-quasi-Hamiltonian path in~$D$.
	
	\textbf{Case 2},~$q\leq3$:
	Recall that dominators are vertices in~$N^-(y)-\{x,v_\ell\}$.
	We distinguish the case where \(P\) contains a dominator and the case where \(P\) does not contain a dominator.
	
	\textbf{Case 2.1},~$V(P)$ does not contain any dominators:
	Let~$w$ be an arbitrary dominator.
	Assume that~$w$ only has positive neighbors in~$V(P)-y$.
	Let~$Q=v_i\dots w$ be a shortest path from~$V(P)$ to~$w$ in~$D-y$.
	Then $Q$ must have internal vertices.
	If~$z\in V(Q)$, merging~$P$ and~$v_0\dots v_{i-1}Qy$ using \cref{lem:merging} yields the desired path.
	Otherwise,~$v_0\dots v_{i-1}Qv_{i+1}\dots v_{\ell+1}$ or~$v_0\dots v_{i-1}Qv_{i+2}\dots v_{\ell+1}$ is a longer~$(x,y)$-quasi-Hamiltonian path in~$D-z$.
	
	Now assume~$w$ has a negative neighbor in~$V(P)-y$.
	If~$zw\in A$,~$xv_1\dots v_{\ell}zwy$ is the desired path.
	If~$xw\in A$,~$xwzv_1\dots v_{\ell}y$ is the desired path.
	If~$v_1w\in A$,~$xv_1wzv_2\dots v_{\ell}y$ is the desired path.
	Let~$j$ be minimal with the property that~$wv_j\in A$ and there exists~$i<j$ such that~$v_iw\in A$.
	The value~$\ell+1$ satisfies this property, so~$j$ is well-defined.
	If~${H(v_{j-1})\neq H(w)}$,~$v_0\dots v_{j-1}wv_j\dots v_{\ell}y$ would be a longer~$(x,y)$-QHP in~$D-z$, so~$H(w)=H(v_{j-1})$.
	Minimality of~$j$ implies~$v_{j-2}w\in A$.
	
	If \(j=3\),~$P'\coloneq v_0v_1wv_3\dots v_\ell y$ is an~$(x,y)$-quasi-Hamiltonian path in~$D-z$ of equal length.
	By \cref{claim:second-vertex-arc-to-z}, we may assume \(zw\in A\), which has already been excluded previously, a contradiction.
	
	If~$3\neq j<\ell+1$,~$P'\coloneq v_0\dots v_{j-2}wv_j\dots v_{\ell}y$ is an~$(x,y)$-quasi-Hamiltonian path in~$D-z$ of equal length that contains a dominator.
	By a previous argument,~$w\rightarrow x$, so~$1\leq j-2$.
	As~\(j\neq 3\),~\(j\) must be at least four.
	Consequently, the path \(P'\) begins with the prefix \(xv_1v_2v_3\).
	In particular, one of the first four vertices of \(P'\) lies in \(R_k\).
	Additionally, the second vertex of \(P'\) is still \(v_1\in R_1\) and the penultimate vertex of~$P'$ remains~$v_\ell\in R_k$.
	Thus, we may restart the proof with~$P'$ instead of~$P$.
	The proof will then finish in case 2.2, so there is no circular reasoning.
	
	If~$j=\ell+1$,~$xv_1v_{\ell}zv_2\dots v_{\ell-1}wy$ is the desired path unless~$H(v_1)=H(v_\ell)$.
	As the arc~$v_{\ell-1}w$ exists,~${w\in R_k}$.
	The previous analysis precludes the arc~$v_1w$, so~$H(v_1)=H(w)$.
	As~$D-y$ is 2-strong, \(w\) has a predecessor \(w'\neq v_{\ell-1}\).
	If~$w'=v_i\in V(P)$, then~${1<i<\ell-1}$ as~$w\Rightarrow\{x,v_1,v_\ell\}$.
	Thus, one of~$xv_1v_{i+1}\dots v_{\ell}zv_2\dots v_iwy$ or~$xv_1v_{i+2}\dots v_{\ell}zv_2\dots v_iwy$ is the desired path.
	Otherwise,~${w'\notin V(P)}$.
	There must be an arc between~$v_1$ and~$w'$.
	If~${v_1w'\in A}$,~$xv_1w'wzv_2\dots v_{\ell}y$ is an~$(x,y)$-QHP.
	Else,~$w'\in R_1$ and there exists a path~$Q'$ from~$v_1$ to~$w'$ in~$R_1$.
	If~${v_2\notin V(Q')}$, then~$xQ'wzv_2\dots v_{\ell}y$ is the desired path.
	Otherwise,~${v_2\in V(Q')\subseteq R_1}$ and~$xv_1v_3\dots v_{\ell-1}zv_2v_{\ell}y$ or~$xv_1v_4\dots v_{\ell-1}zv_2v_{\ell}y$ is the desired path.
	
	\textbf{Case 2.2},~$V(P)$ contains a dominator:
	Let~$1\leq i<l$ be maximal such that~$v_i$ is a dominator.
	We differentiate the cases~$i\neq1$ and~$i=1$.
	Assume~$i\neq 1$.
	Then one of~$xv_1v_{i+1}\dots v_{\ell}zv_2\dots v_iy$ or~$xv_1v_{i+2}\dots v_{\ell}zv_2\dots v_iy$ is the desired path unless~$H(v_1)=H(v_\ell)$ and~$i=\ell-1$.
	If there exists a~\(j\) with~$2\leq j<\ell-1$ and~$v_jv_\ell\in A$, then~$xv_1v_{j+1}\dots v_{\ell-1}zv_2\dots v_jv_{\ell}y$ or~$xv_1v_{j+2}\dots v_{\ell-1}zv_2\dots v_jv_{\ell}y$ is the desired path.
	Otherwise, it follows that~$v_\ell\Rightarrow\{v_1,\dots,v_{\ell-2}\}$.
	As~$y$ is not contained in any 2-cycles, we obtain~$v_\ell\rightarrow \{y,z\}$, but~$D-x$ is 2-strong, so there exists a vertex~${u\in V-\{x,y,z,v_{\ell-1}}\}$ such that~$uv_\ell\in A$.
	Let~$Q=v_j\dots u$ be a shortest path from~$V(P)$ to~$u$ in~$D-\{y,v_\ell\}$.
	If~${z\in V(Q)}$,~$P[x,v_j[Qv_{\ell}y$ and~$P$ can be merged to obtain the desired path, so assume this is not the case.
	If~${j=l-2}$, then~$xv_1v_{\ell-1}zv_2\dots v_{\ell-3}Qv_{\ell}y$ is the desired path.
	If~${j=\ell-1}$,~$xv_1\dots v_{\ell-2}Qv_{\ell}y$ is a longer~$(x,y)$-QHP in~$D-z$, which is a contradiction.
	If~$j=0$,~$Qv_{\ell}zv_1\dots v_{\ell-1}y$ is the desired path.
	If~$0<j<l-2$,~$P[x,v_j[Qv_\ell v_{j+1}\dots v_{\ell-1}y$ is a longer~$(x,y)$-QHP in~$D-z$ unless~$H(v_{j+1})=H(v_\ell)$.
	Then~$xv_1v_{j+2}\dots v_{\ell-1}zP[v_2,v_j[Qv_{\ell}y$ is the desired path.
	
	Now assume~$i=1$.
	As~$y$ has a positive neighbor in~$R_1$ and does not lie on a 2-cycle,~$|R_1|>1$.
	Additionally,~$R_1$ must be strong as otherwise~$N^-(v_1)\subseteq\{x,y,z\}$ which contradicts the fact that~$D-x$ is 2-strong and~$y$ does not lie on a 2-cycle.
	Let~$u\in R_1\cap N^+(v_1)$.
	Note the it may be that \(u=v_2\).
	If the arc~$uv_2$ exists, we obtain a longer~$(x,y)$-QHP in~$D-z$.
	If the arc~$v_2u$ exists, we obtain~$v_2\in R_1$ and one of~$xv_1v_3\dots v_{\ell-1}zv_2v_{\ell}y$,~$xv_1uv_3\dots v_{\ell-1}zv_2v_{\ell}y$,~$xv_1v_3\dots v_{\ell-1}zv_2uv_{\ell}y$ or~$xv_1v_4\dots v_{\ell-1}zv_2uv_{\ell}y$ is the desired path.
	Thus, there can be no arcs between~$v_2$ and~$u$, so~\(H(u)=H(v_2)\).
	Then~$xv_1uv_3\dots v_{\ell}y$ is an~$(x,y)$-QHP in~$D-z$ with the same length as~$|P|$, so we may assume~$zu\in A$ by previous arguments.
	If \({H(u)\neq H(v_\ell)}\), one of~$xv_1v_3\dots v_{\ell-1}zuv_{\ell}y$ or~$xv_1v_4\dots v_{\ell-1}zuv_{\ell}y$ is the desired path, so we may assume this is not the case.
	Thus, the vertex \(u\) is a vertex in \(R_1\) with \(H(u)=H(v_\ell)\) that behaves similarly to \(v_2\) and potentially even is~\(v_2\).
	
	If there exists~$2\leq j<\ell-2$ such that~$v_jv_\ell\in A$, one of the paths~$xv_1v_{j+1}\dots v_{\ell-1}zv_2\dots v_jv_{\ell}y$ or~$xv_1v_{j+2}\dots v_{\ell-1}zv_2\dots v_jv_{\ell}y$ is the desired path.
	If there exists~$2\leq j<\ell-2$ with~${v_jz\in A}$, then one of the paths~$xv_1v_3\dots v_jzuv_{j+1}\dots v_{\ell}y$,~$xv_1v_4\dots v_jzuv_{j+1}\dots v_{\ell}y$,~$xv_1v_3\dots v_jzuv_{j+2}\dots v_{\ell}y$ or~$xv_1v_4\dots v_jzuv_{j+2}\dots v_{\ell}y$ is the desired path.
	Thus, assume no such~$j$ exists.
	As~$i=1$ is maximal, this implies~$\{y,z,v_\ell\}\Rightarrow\{v_2,\dots,v_{\ell-3}\}$.
	If the arc \(v_{\ell-2}v_\ell\) exists, one of \(xv_1v_{\ell-1}zv_2\dots v_{\ell-2}v_\ell y\) or \(xv_1uv_{\ell-1}zv_3\dots v_{\ell-2}v_\ell y\) is the desired path.
	Otherwise, it is impossible that~$R_k\subseteq\{v_2,\dots,v_\ell\}$ as the vertex~$v_{\ell-2}$ separates~$v_3\in R_k$ from~$v_1$ in~$D-x$ in this case contradicting the fact that~\(D-x\) is 2-strong.
	Therefore, there exists~$w\in R_k-\{v_2,\dots,v_\ell\}$ such that~$wz\in A$ or~$wy\in A$.
	In the first case,~$xv_1wzv_2\dots v_{\ell}y$ or~$xv_1uwzv_3\dots v_{\ell}y$ is the desired path.
	Otherwise,~$z\rightarrow w$ and~$xv_1\dots v_{\ell}zwy$ is the desired path.
	
	\textbf{Case 3},~$3<\ell-1\leq q$:
	If~$|R_1\cap V(P)|>2$, reversing all arcs and applying the previous arguments to the reversal of~$P$ yields a~$(y,x)$-QHP in the reversed graph, i.e.~the desired path in~$D$.
	Thus, we may assume that~$R_1\cap V(P)\subseteq\{v_1,v_2\}$.
	Assume~$v_1v_{\ell-1}\in A$.
	Then the path~$xv_1v_{\ell-1}zv_2\dots v_{\ell-2}v_{\ell}y$ or~$xv_1v_{\ell-1}zv_2\dots v_{\ell-3}v_{\ell}y$ is the desired path unless~$l=4$ and~${H(v_2)=H(v_4)}$.
	However, this is impossible as~$\chi(D-\{x,y,z\})\geq4$, so we may assume that~$v_1v_{\ell-1}\notin A$.
	This implies~$H(v_1)=H(v_{\ell-1})$ and~$l>4$.
	Then~$H(v_1)\neq H(v_{\ell-2})$, so~${v_1v_{\ell-2}\in A}$.
	If~$H(v_{\ell-3})\neq H(v_\ell)$,~$xv_1v_{\ell-2}v_{\ell-1}zv_2\dots v_{\ell-3}v_{\ell}y$ is the desired path.
	Otherwise,~${l>5}$ as~$\chi(D-\{x,y,z\})\geq4$,~$H(v_1)=H(v_{\ell-1})$ and~$H(v_{\ell-3})=H(v_\ell)$.
	Thus, the path~$xv_1v_{\ell-2}v_{\ell-1}zv_2\dots v_{\ell-4}v_{\ell}y$ is an \((x,y)\)-quasi-Hamiltonian path in \(D\).
\end{proof}

For semicomplete digraphs, the chromatic number is identical to the number of vertices.
Thus, the condition \(\chi(D-\{x,y,z\})\geq4\) simplifies to \(|V|\geq7\).
A semicomplete digraph satisfying all other conditions of \cref{thm:not-2-strong-without-xy} has at least 6 vertices, so restricting \cref{thm:not-2-strong-without-xy} to semicomplete digraphs almost exactly yields \cite[Theorem~3.3]{DBLP:journals/jal/Bang-JensenMT92}, which does not have any restriction on the number of vertices.

\rstMtThree
\begin{proof}
	As this only makes finding an~$(x,y)$-QHP harder, we may assume that~$H(x)$ and~$H(y)$ are singletons by adding the arcs~$vx$ or~$yv$ whenever necessary while ensuring that~$y\rightarrow x$.
	If no~2-separators exist, we obtain the desired path by \cref{thm:3-paths}.
	By \cref{thm:not-2-strong-without-xy}, we may assume that~$D-\{x,y\}$ is 2-strong.
	Let~$u,w\in V$ be a 2-separator of~$x$ and~$y$.
	Assume that~$uwu$ is a 2-cycle using \cref{lem:2-sep-arcs}.
	We will assume that~$N^-(y)=\{u,w\}$ as the case~$N^+(x)=\{u,w\}$ is symmetric.
	
	Assume there is an arc from~$x$ to~$\{u,w\}$, without loss of generality let~$xu\in A$.
	If~$xw\in A$, let~$Q$ be a~$[u,w]$-QHP in~$D-\{x,y\}$ as given by \cref{thm:weak-qhp}.
	\Cref{thm:weak-qhp} is applicable as we have
	\begin{equation*}
		\chi(D-\{x,y,u,w\})\geq\left\lceil\frac{|V|-4}{\alpha(D)}\right\rceil\geq6.
	\end{equation*}
	Then~$xQy$ is the desired path.
	Otherwise,~$xw\notin A$.
	If there are 2 internally disjoint~$(x,w)$-paths in~$D-\{y,u\}$, there are~3 internally disjoint~$(x,w)$-paths of length at least 2 in~$D-y$ and so \cref{thm:3-paths} yields an~$(x,w)$-QHP in~$D-y$ and with that an~$(x,y)$-QHP in~$D$.
	Thus, there exists~$r\in V-\{x,y,u,w\}$ such that~$\{u,r\}$ is a 2-separator of~$x$ and~$w$ in~$D-y$.
	As this is also a 2-separator of~$x$ and~$y$ in~$D$, it must be a trivial 2-separator, so~$N^+(x)=\{u,r\}$.
	By \cref{lem:2-sep-arcs}, we may assume~$uru$ to be a 2-cycle in~$D$.
	If~$rw\notin A$, there must be 2 internally disjoint~$(r,w)$-paths in~$D-\{x,y,u\}$ because all 2-separators of~$x$ and~$y$ are trivial.
	With the third path~$ruw$, this yields an~$(r,w)$-QHP in~$D-\{x,y\}$ through \cref{thm:3-paths} and with that an~$(x,y)$-QHP in~$D$, so assume~$rw\in A$.
	
	Set~$D'\coloneq D-\{x,y\}$.
	If~$D'-r$ contains a~$[u,w]$-QHP, we obtain the desired path, so assume this is not the case.
	Thus, one of the two conditions of \cref{thm:weak-qhp} must apply as we have
	\begin{equation*}
		\chi(D-\{x,y,r,u,w\})\geq\left\lceil\frac{|V|-5}{\alpha(D)}\right\rceil\geq5.
	\end{equation*}
	If the first condition applies,~$D'-\{r,u\}$ is not strong.
	As~$D'$ is 2-strong, there exists an~$(r,u)$-QHP~$Q$ in~$D'$ by \cref{lem:not-2-strong} and~$xQy$ is the desired path.
	If the alternative condition applies,~${D'-\{r,w\}}$ is not strong and contains an~$(r,w)$-QHP in~$D'$ by a symmetric argument, which also leads to the desired path.
	This completes the proof if there is an arc from~$x$ to~$\{u,w\}$, so we may assume~$\{u,w\}\Rightarrow x$.
	
	If there are~3 internally disjoint~$(x,\{u,w\})$-paths in~$D-y$, there are~3 internally disjoint~$(x,u)$-paths or~3 internally disjoint~$(x,w)$-paths of length at least 2, so \cref{thm:3-paths} yields an~$(x,y)$-QHP.
	Therefore, we may assume that there are~$r,s\in V-\{x,y,u,w\}$ such that~$\{r,s\}$ is a~2-separator of~$x$ and~$\{u,w\}$.
	Then~$\{r,s\}$ also separates~$x$ and~$y$, so~${N^+(x)=\{r,s\}}$.
	As before, assume that~$srs$ is a 2-cycle in~$D$.
	If~$\{u,w\}\Rightarrow\{r,s\}$, then the fact that all~$(x,y)$-separators of size~2 are trivial implies that there are~3 internally disjoint~$(\{r,s\},\{u,w\})$-paths in~$D-\{x,y\}$.
	Thus, there are~3 internally disjoint~$(v_0,v_1)$-paths for some~$v_0\in\{r,s\}$ and~$v_1\in \{u,w\}$ in~$D-\{x,y\}$ completing the proof using \cref{thm:3-paths}.
	Without loss of generality, assume that~$ru\in A$.
	If~$\{u,w\}\Rightarrow s$, there are~3 internally disjoint~$(s,\{u,w\})$-paths of length at least~2 in~$D-\{x,y\}$ because all 2-separators of~$(x,y)$ are trivial, which again completes the proof through \cref{thm:3-paths}.
	Thus, we may assume that~$su\in A$ or~$sw\in A$.
	
	Consider the case~$su\in A$.
	If~$D-\{x,y,u\}$ contains an~$[r,s]$-QHP~$Q$,~$xQuy$ is the desired path.
	Otherwise, \cref{thm:weak-qhp} implies that~$D-\{x,y,u,r\}$ or~$D-\{x,y,u,s\}$ is not strong.
	If~$D-\{x,y,u,r\}$ is not strong, \cref{lem:not-2-strong} yields an~$(r,u)$-QHP~$Q$ in~$D-\{x,y\}$ as~$D-\{x,y\}$ is~2-strong.
	If~$D-\{x,y,u,s\}$ is not strong, \cref{lem:not-2-strong} yields an~$(s,u)$-QHP~$Q$ in~$D-\{x,y\}$.
	In either case,~$xQy$ is the desired path.
	If~$rw,sw\in A$, we argue similarly to obtain an~$(x,y)$-QHP.
	Thus, we may assume that the only arcs from~$\{r,s\}$ to~$\{u,w\}$ are~$ru$ and~$sw$.
	
	Set~$R\coloneq D-\{x,y,r,s,u,w\}$.
	As~$D-\{x,y\}$ is 2-strong, there are at least~2 arcs in each direction between~$R$ and~$\{r,s,u,w\}$.
	Assume~$R$ is not strong with linear decomposition~$R_1,\dots,R_k$.
	Then there exists an~$(a,b)$-QHP in~$R$ for any pair of vertices~$a\in R_1$ and~$b\in R_k$ by \cref{lem:non-strong}.
	If there exists~${v_0\in\{r,s\}}$ and~${v_1\in\{u,w\}}$ such that~$N^+(v_0)\cap R_1\neq\emptyset$ and~$N^-(v_1)\cap R_k\neq\emptyset$, there exists a~$(v_0,v_1)$-QHP in~${D[V(R)\cup\{v_0,v_1\}]}$.
	This path easily expands to an~$(x,y)$-QHP in~$D$.
	Thus,~$R_1\Rightarrow\{r,s\}$ or~$\{u,w\}\Rightarrow R_k$.
	In the first case, both~$u$ and~$w$ must have positive neighbors in~$R_1$ as~$D-\{x,y\}$ is 2-strong.
	A vertex in~$R_k$ must have a positive neighbor~$v_1\in\{r,s,u,w\}$.
	No matter which one it is, it can be combined with some~$v_0\in\{u,w\}$ to obtain a~$(v_0,v_1)$-QHP in~$D[V(R)\cup\{v_0,v_1\}]$ which extends to an~$(x,y)$-QHP in~$D$.
	The case~$\{u,w\}\Rightarrow R_k$ is symmetric.
	Therefore, we may assume that~$R$ is strong, so there exists a quasi-Hamiltonian cycle in~$R$ by \cref{lem:cycle-cover}.
	
	Let~$C=v_0\dots v_{\ell-1} v_0$ be a quasi-Hamiltonian cycle in~$R$.
	If~$\{r,s\}\Rightarrow V(C)$, let~$P=v_i\dots a$ be a shortest~$(V(C),\{r,s\})$-path in \(D-\{x,y\}\).
	As \({D-\{x,y\}}\) is 2-strong, both \(u\) and \(w\) must have negative neighbors in \(C\) so that \(P\) consists of~3 vertices, i.e.~\(P=v_iba\) with \({b\in\{u,w\}}\).
	Let~\(c\in\{r,s\}\setminus\{a\}\) and \(d\in\{u,w\}\setminus\{b\}\).
	Then one of~$xcC[v_{i+1},v_i[Pdy$ or~$xcC[v_{i+2},v_i[Pdy$ is the desired path where indices are viewed modulo~$\ell$.
	Thus, we may assume that there exists an arc from~$V(C)$ to~$\{r,s\}$.
	Without loss of generality, assume that~$v_1r\in A$.
	If~$sv_2\in A$ or~$wv_2\in A$, we obtain the~$(x,y)$-QHP~$xsC[v_2,v_1]ruwy$ or~$xswC[v_2,v_1]ruy$, so~$v_2\Rightarrow\{s,w\}$.
	As~$H(s)\neq H(w)$, one of the arcs~$v_2s$ or~$v_2w$ must exist in~$D$.
	Using an analogous argument, we obtain~$v_3\Rightarrow\{r,u\}$.
	Inductively, we conclude that~$v_{2i+1}\Rightarrow\{r,u\}$ and~$v_{2i}\Rightarrow\{s,w\}$ for all~$i\in\N$.
	In particular, \(\ell\) can be assumed to be even as otherwise both properties are satisfied by all vertices of the cycle, which results in the desired path.
	
	If~$V(C)\Rightarrow\{r,s\}$, let~$P=a\dots v_i$ be a shortest~$(\{r,s\},V(C))$-path in~$D-\{x,y\}$.
	Because~\({D-\{x,y\}}\) is 2-strong, both \(u\) and \(w\) must have positive neighbors in \(C\), so that \(P\) consists of~3 vertices, i.e.~\({P=abv_i}\) with \({(a,b)\in\{(s,u),(r,w)\}}\).
	Let \(c\in\{r,s\}\setminus\{a\}\) and \(d\in\{u,w\}\setminus\{b\}\).
	Then,~${xPC]v_i,v_{i-1}]cdy}$ or~${xPC]v_i,v_{i-2}]cdy}$ is the desired path.
	Thus, we may assume that~$r$ or~$s$ has a positive neighbor on~$C$, i.e.~$rv_{2i}\in A$ or~$sv_{2i+1}\in A$ for some~$i\in\N$.
	In either case, an inductive argument analogous to the previous one yields~$\{s,w\}\Rightarrow v_{2i+1}\Rightarrow\{r,u\}$ and~$\{r,u\}\Rightarrow v_{2i}\Rightarrow\{s,w\}$ for all~$i\in\N$.
	
	Assume the set~$R$ contains at least~5 pairwise different color classes.
	Then, by \cref{lem:2-out-of-3-path}, there exists a QHP~$Q=w_1\dots w_\ell$ in~$D[R]$ with end vertices in~$\{v_0,v_2,v_4\}$.
	One of the arcs~$rw_1$ and~$uw_1$ must exist as~$\alpha(D)\leq2$.
	Similarly, one of the arcs~$w_{\ell}s$ and~$w_{\ell}w$ must exist.
	If~$rw_1\in A$, then~$xrQswuy$ or~$xsrQwuy$ is an \((x,y)\)-quasi-Hamiltonian path in \(D\).
	Otherwise,~$uw_1\in A$, so~$xruQswy$ or~$xsruQwy$ is a quasi-Hamiltonian path.

	Finally, assume that~$R$ intersects at most~4 pairwise different color classes.
	Then~${\alpha(D)=1}$ because~$\alpha(D)=2$ would imply~$|V(R)|=|V|-6\geq9$ which yields~${|\C(R)|\geq 5}$.
	The assumption~\({|V|\geq5+5\alpha(D)}\) implies~$|V(R)|\geq4$, so~$C$ consist of exactly 4 vertices.
	As~$\alpha(D)=1$, we have~$v_3\rightarrow\{r,u\}\rightarrow v_0$, so~$xsrv_0v_1v_2v_3uwy$ is a Hamiltonian path in~$D$.
\end{proof}

\Cref{thm:3-paths}, \cref{thm:not-2-strong-without-xy} and \cref{thm:2-sep-all-trivial} restricted to semicomplete digraphs are the structural results used in~\cite{DBLP:journals/jal/Bang-JensenMT92} to construct a polynomial-time algorithm for the Hamiltonian~$(s,t)$-path problem on semicomplete digraphs.
However, it is not clear how one might generalize the algorithmic results of the same work in order to obtain a similar algorithm for quasi-Hamiltonian~$(s,t)$-paths in semicomplete multipartite digraphs with independence number at most~2.
The problem is \NP-complete even when the independence number is bounded by~3 as shown in \cref{sec:covering-path}, but the results of this section suggest that the problem may be polynomial-time solvable for independence number at most~2.

\subsection{Characterizing Hamiltonian Paths}
The existence of Hamiltonian paths in general semicomplete multipartite digraphs is characterized by the existence of a 1-path-cycle factor as shown by Gutin~\cite{DBLP:journals/siamdm/Gutin93}.
Restricting the graph class to those semicomplete multipartite digraphs with independence number at most~2 allows a simpler characterization in terms of unilateral connectivity.
To prove this, we use the following two results:
\begin{theoremq}{\cite{fink1980traceable}}\label{thm:unilateral-to-traceable}
	Let~$n\in\N$.
	If~$\mathcal{P}$ is a collection of vertex-disjoint paths of length at most 2 in~$K_n$, then every unilaterally connected orientation of~$K_n-\bigcup_{P\in\mathcal{P}}E(P)$ contains a Hamiltonian path.
\end{theoremq}

\begin{theoremq}[{\cite[Corollary~2.4]{DBLP:journals/arscom/Volkmann01}}]\label{thm:spanning_mtournament}
	A strong semicomplete multipartite digraph with at least three color classes contains a spanning strong multipartite tournament.
\end{theoremq}

\begin{theorem}\label{thm:hamiltonian-criterion-nsmd}
	Let~$D=(V,A)$ be a semicomplete multipartite digraph such that~$\alpha(D)\leq2$.
	Then~$D$ is unilaterally connected if and only if~$D$ contains a Hamiltonian path.
\end{theorem}
\begin{proof}
	Clearly, \(D\) is unilaterally connected if it contains a Hamiltonian path, so let \(D\) be unilaterally connected.
	We will show that \(D\) contains a spanning unilaterally connected tournament by induction over the number of 2-cycles.
	By \cref{thm:unilateral-to-traceable}, this implies that \(D\) contains a Hamiltonian path.
	If \(D\) does not contain any 2-cycles, \(D\) already is a multipartite tournament.
	
	Let \(R_1,\dots,R_k\) be the linear decomposition of \(D\).
	As \(D\) is unilaterally connected, every linear component either induces a strong component or is a singleton set.
	Due to the structure of the linear decomposition, every 2-cycle is contained in some linear component.
	Let \(R_i\) be a linear component containing a 2-cycle.
	If \(|R_i|\geq5\), then we may assume that \(D[R_i]\) is a strong multipartite tournament by \cref{thm:spanning_mtournament}.
	This leaves only the cases where \(R_i\) intersects exactly two color classes, so \(2\leq|R_i|\leq4\).
	By inserting two new vertices, one with outgoing arcs to all other vertices and one with incoming arcs from all other vertices, we may without loss of generality assume that \(1<i<k\).
	
	If \(|R_i|=2\), let \(R_i=\{v,w\}\).
	First assume \(R_{i-1}\rightarrow R_i\).
	By \cref{lem:linear-decomp}, one of \(v\) or \(w\) has an outgoing arc to \(R_{i+1}\), say \(w\).
	Then \(D-wv\) is unilaterally connected with one fewer 2-cycle, so it contains a spanning unilaterally connected tournament by the induction hypothesis.
	Otherwise, exactly one of \(v\) or \(w\) does not have an incoming arc from \(R_{i-1}\), say \(w\).
	Then \(w\) must have an outgoing arc to \(R_{i+1}\) as the unique vertex with the same color as \(w\) lies in \(R_{i-1}\).
	Again, \(D-wv\) is unilaterally connected, so it contains a spanning unilaterally connected tournament by the induction hypothesis.
	
	If \(|R_i|=3\), let \(R_i=\{u,v,w\}\) with \(H(u)=H(w)\).
	Because \(D[R_i]\) is strong, it contains a~\((u,w)\)-path \(P\), which must be \(uvw\).
	Remove all arcs from \(D\) that run against the direction of \(P\), i.e.~the arcs \(vu\) and \(wv\).
	As all color classes have size at most~2, \(R_{i-1}\rightarrow\{u,w\}\rightarrow R_{i+1}\).
	Thus, \(D-\{vu,wv\}\) is unilaterally connected, so it contains a spanning unilaterally connected tournament by the induction hypothesis.
	
	If \(|R_i|=4\), let \(R_i=\{v_1,v_2,v_3,v_4\}\) with \(H(v_1)=H(v_3)\) and \(H(v_2)=H(v_4)\).
	If \(D[R_i]\), contains a Hamiltonian path \(P\), remove all arcs from \(D\) that run against the direction of \(P\).
	As \(R_{i-1}\rightarrow R_i\rightarrow R_{i+1}\), the resulting digraph is still unilaterally connected, so it contains a spanning unilaterally connected tournament by the induction hypothesis.
	Otherwise, let \(P\) be a~\((v_1,v_j)\)-path.
	Without loss of generality, assume that \(P=v_1v_2v_3\).
	Because \(D[R_i]\) does not contain a Hamiltonian path, \(v_3v_4,v_4v_1\notin A\).
	This yields the \((v_1,v_3)\)-path \(Q=v_1v_4v_3\).
	Because~\(D[R_i]\) does not contain a Hamiltonian path, \(v_2v_1,v_3v_2\notin A\).
	However, this contradicts the assumption that \(R_i\) contains a 2-cycle.
\end{proof}

\Cref{thm:hamiltonian-criterion-nsmd} is best possible in the sense that the condition \(\alpha(D)\leq2\) cannot be relaxed to~\({\alpha(D)\leq3}\).
This fact can be verified by considering the complete bipartite graph~$K_{1,3}$:
This graph does not contain a Hamiltonian path, so no biorientation of it can contain one either.
However, when replacing each edge with two opposing arcs, the result is a strongly connected semicomplete bipartite digraph.
As tournaments are unilaterally connected by definition, \cref{thm:hamiltonian-criterion-nsmd} generalizes Rédei's result stating that every tournament contains a Hamiltonian path~\cite{redei1934kombinatorischer}.

\begin{corollary}\label{cor:Hamiltonian-alg}
	The Hamiltonian path problem can be solved in time~$\Oh(n^2)$ for semicomplete multipartite digraphs with independence number at most 2.
\end{corollary}
\begin{proof}
	A semicomplete multipartite digraph is unilaterally connected if and only if no linear component is an independent set of size at least~2.
	As the input graph \(D\) has independence number at most~2, \(D\) is unilaterally connect if and only if every linear component of size~2 is a~2-cycle.
	The linear decomposition can be computed in time \(\Oh(n^2)\)~\cite{DBLP:journals/dam/Bang-JensenMS13}, so this condition can also be checked in time \(\Oh(n^2)\).
\end{proof}

\Cref{cor:Hamiltonian-alg} slightly improves Gutin's algorithm~\cite{DBLP:journals/siamdm/Gutin93} which runs in time~$\Oh(n^\beta)$ indicating that the bounded independence number makes the class less complex from an algorithmic perspective.
Here,~$\beta$ is the exponent in the running time of the algorithm for computing a minimum cost maximum bipartite matching.
The best currently known value, \(\beta\leq2+\oh(1)\), is due to Chen et.~al.~\cite{DBLP:journals/cacm/ChenKLPGS23}.
While slower, the algorithm from~\cite{DBLP:journals/siamdm/Gutin93} works for SMDs with arbitrary independence number.